%% file: transitivity_horseshoe_2-torus_ressub_v2.tex
\documentclass[10pt]{amsart}
\usepackage[T1]{fontenc}
\usepackage[english]{babel}
\usepackage[latin1]{inputenc}
\usepackage{bm}
\usepackage{amsmath}
\usepackage{amsthm}
\usepackage{amssymb}
\usepackage{amscd}
\usepackage{amsfonts}
\usepackage{amsbsy}
\usepackage{mathtools,cancel}
\usepackage{array}
\usepackage{graphicx}
\usepackage{xcolor} 
\usepackage{tikz-cd}    
\usepackage{setspace}                   
\usepackage{indentfirst}                
\usepackage[fixlanguage]{babelbib}
\usepackage[top=3.0cm,bottom=3.0cm,left=3.0cm,right=3.0cm]{geometry} 
\usepackage{hyperref} 

\usepackage{enumitem}

\theoremstyle{plain}
\newtheorem{theorem}{Theorem}[section]     		
\newtheorem*{theorem*}{Theorem}
\newtheorem*{theoremA}{Theorem A}
\newtheorem*{theoremB}{Theorem B}

\newtheorem{corollary}[theorem]{Corollary}
\newtheorem{proposition}[theorem]{Proposition}
\newtheorem{lemma}[theorem]{Lemma}
\newtheorem{affirmation}[theorem]{Claim}

\theoremstyle{definition}
\newtheorem{definition}[theorem]{Definition}
 
\newtheorem{notation*}[theorem]{Notation}

\theoremstyle{remark}            	
\newtheorem{remark}[theorem]{Remark}

\newcommand{\T}{\mathbb{T}}
\newcommand{\A}{\mathbb{A}}
\newcommand{\R}{\mathbb{R}}
\newcommand{\Q}{\mathbb{Q}}
\newcommand{\Z}{\mathbb{Z}}
\newcommand{\N}{\mathbb{N}}
\newcommand{\I}{{\mathcal{I}}}

\newcommand{\F}{{\mathcal{F}}}

\newcommand{\tildegamma}{\tilde{\gamma}}

\newcommand{\tildeF}{\tilde{\mathcal{F}}}

\newcommand{\fix}{\operatorname{fix}}
\newcommand{\sing}{\operatorname{sing}}
\newcommand{\dom}{\operatorname{dom}}
\newcommand{\deck}{\operatorname{deck}}

\newcommand{\inte}{\operatorname{B}_{\operatorname{c}}}
\newcommand{\exte}{\operatorname{UB}_{\operatorname{c}}}
\newcommand{\pr}{\operatorname{p}}
\newcommand{\dist}{\operatorname{d}}
\newcommand{\sphere}{\operatorname{sphere}}
\newcommand{\homeo}{\operatorname{Homeo}}

\newcommand{\rot}{\operatorname{rot}}

\newcommand{\tildedom}{\widetilde{\dom}}
\newcommand{\Ne}{\operatorname{ne}}

\renewcommand{\S}{\mathbb{S}}

\definecolor{azul}{rgb}{0,0,1}
\definecolor{vermelho}{rgb}{1,0,0}

\title{Transitivity and the existence of horseshoes on the $2$-torus}
\author{Pollyanna Vicente Nunes}
\address[Nunes]{Instituto de Matem\' atica e Estat\' istica da Universidade de S\~ao Paulo,
	R. do Mat\~ ao, 1010 - Butant\~a, S\~ ao Paulo, Brasil - 05508---090}
\email{pollyannavnunes@gmail.com }

\author{F\' abio Armando Tal}
\address[Tal]{Instituto de Matem\' atica e Estat\' istica da Universidade de S\~ao Paulo,
	R. do Mat\~ ao, 1010 - Butant\~a, S\~ ao Paulo, Brasil - 05508---090} 
\email{fabiotal@ime.usp.br}

\thanks{The first author was supported by CAPES - Finance code 001. The second author was partially supported by CNPq and FAPESP}

\begin{document}
	\begin{abstract}
		\noindent We study the relationship between transitivity and topological chaos for homeomorphisms of the two torus. We show that if a transitive homeomorphism of $\T^2$ is homotopic to the identity and has both a fixed point and a periodic point which is not fixed, then it has a topological horseshoe. We also show that if a transitive homeomorphims of $\T^2$ is homotopic to a Dehn twist, then either it is aperiodic or it has a topological horseshoe. 
	\end{abstract}
	
	\maketitle

	\section{Introduction}
	
	Three topological concepts that attempt to describe the complexities of dynamical systems are the notions of transitivity, of positive topological entropy, and of exponential growth of periodic points. While these three concepts are independent and there are several examples of systems possessing one of them but not the other two, it is a consequence of a result by Katok (see \cite{katok80}) that, for $\mathcal{C}^{1+\alpha}$ diffeomorphisms of surfaces positive topological entropy implies the exponential growth of periodic points, since positive topological entropy implies the presence of horseshoes. On the other hand, transitivity is usually a ``weaker'' requirement, not implying the richness of dynamical behavior that positive topological entropy does. 
	
	But while in general transitivity does not imply that a system is topologically chaotic (i.e., has both positive topological entropy and also exponential growth of periodic points), there is evidence to suggest that, for surface homeomorphims, these concepts are related, and that in many cases the former implies the later. Indeed, it is a direct consequence of Corollary J of \cite{ferradura}, improving on both a theorem by Handel (\cite{handel92}) and previous work by the same authors (\cite{forcing}), that every transitive homeomorphism  of the $2$-sphere {with three distinct periodic orbits} is topologically chaotic. Note that this is optimal in some ways. There are transitive $\mathcal{C}^{\infty}$ diffeomorphisms of $\S^2$ having only two fixed points, and every non-wandering homeomorphism of $\S^2$ has at least two fixed points.
	
	Our goal here is to try and extend the previous result, as best as possible, to different surfaces and in particular to $\T^2$. Since this paper is concerned with continuous dynamics without any other smoothness requirement, let us introduce a purely topological concept of horseshoe, a variation of the one proposed by Kennedy-Yorke (\cite{ky01}) that appeared in (\cite{ferradura}). 
	
	\begin{definition}[Topological Horseshoe]\label{def:horseshoe}
		Given $X$ a Hausdorff locally compact topological space and $f:X\to X$ a homeomorphism on $X$, we say that a compact subset $\Delta$ of $X$ is a \textit{topological horseshoe for $f$} if $\Delta$ is invariant by a power $f^q$ of $f$ and if there exist a Hausdorff compact space $Y$, a continuous map $g: Y\to Y$, a integer $r\geq 2$ and two continuous and surjective maps (factor maps) $\pi_1: Y \to \Delta$ and $\pi_2: Y \to \Sigma$ where $\Sigma=\{0,\cdots,r-1\}^{\Z}$, such that:  
		\begin{enumerate}
			\item the factor map $\pi_1$ is a semi-conjugation between $g$ and $f^q|_{\Delta}$, where fibers of $\pi_1$ are all finite with a uniform bound in their cardinality;
			\item the factor map $\pi_2$ is a semi-conjugation between ${g}$ and the Bernoulli shift, $\sigma:\Sigma\to \Sigma$;
			\item The preimage of every $n$-periodic sequence of $\Sigma$, by the factor map $\pi_2$, contains a $n$-periodic point of ${g}$.
		\end{enumerate}
	\end{definition}
	
	It is an immediate consequence of Definition \ref{def:horseshoe} that any homeomorphism of a compact metric space possessing a topological horseshoe has both strictly positive topological entropy, and also that some power of the map has exponential growth of periodic points, so that the system is topologically chaotic.
	
	In this paper we investigate to what {extent} does transitivity imply the existence of a topological horseshoe for homeomorphisms of the $2$-torus.  
	
	There are of course examples of transitive torus homeomorphims without periodic points, like ergodic rotations, but there are even examples of transitive torus homeomorphisms that have positive topological entropy and no periodic points (\cite{rees81}). On the other hand, the same kind of restriction found for sphere homeomorphisms does not hold. One can build examples of $2$-torus homeomorphims with arbitrarily many fixed point that are transitive and smooth, and with zero entropy. This can be done by taking the time one map of a flow of the type $x\prime = \phi(x)  v$, where $v$ is a totally irrational vector, and $\phi:\T^2\to [0,1]$ is null at  a given finite set of points. The zeros of $\phi$ will be fixed points of the dynamics and careful choice of $\phi$ will also yield that the dynamic is transitive. Furthermore, a similar procedure could be done and one could obtain that the dynamics is even area-preserving, see \cite{salvadortal}.
	
	Further complicating the picture for torus homeomorphisms is that these do not all belong to the same isotopy class. The standard classification of isotopy classes of homeomorphims of $\T^2$ says that, if $f$ is in $\homeo(\T^2)$, then $f$ satisfies one of the following:
	\begin{enumerate}
		\item A power of $f$ is isotopic to the identity,
		\item A power of $f$ is isotopic to a Dehn twist,
		\item $f$ is isotopic to a linear Anosov homeomorphism.
	\end{enumerate}
	
	It is a well known fact that any homeomorphism of $\T^2$ that is isotopic to a linear Anosov map has a topological horseshoe. So we concentrate on the two models cases for the other situations. 
	
	For the homotopic to identity case we have the following:
	
	\begin{theoremA}
		Let $f$ be a transitive homeomorphism of $\T^2$ and let $g$ be a power of $f$ such that $g$ is isotopic to the identity. If $g$ has both a fixed point and a non-fixed periodic point, then $f$ has a topological horseshoe.
	\end{theoremA}
	
	{
		Note that Theorem A also holds if we just assume that $g$ has two periodic points of different periods, instead of assuming it has a fixed. Indeed, let $z_0,z_1\in\T^2$ be $q_0, q_1$-periodic points, respectively, such $q_0,q_1>1$ with $q_0\neq q_1$. If $h:=g^{q_0}$ then $h$ satisfy the hypothesis of the Theorem A and so we get a horseshoe for $f$.

		The previous theorem and examples dealt with homeomorphisms homotopic to the identity. For homeomorphisms homotopic to a Dehn twist case, we also have transitive examples  without periodic points. For instance, a well known example is the map $(x,y)\to(x+y,y+\alpha)$ where $\alpha\in\R\setminus\Q$, which is a minimal homeomorphism. Our next result proves that a transitive homeomorphism homotopic to a Dehn twist is either aperiodic or it has a topological horseshoe.}
	
	\begin{theoremB}
		Let $f$ be a transitive homeomorphism of $\T^2$ and let $g$ be a power of $f$ such that $g$ is isotopic to a Dehn twist. If $g$ has a periodic point, then $f$ has a topological horseshoe.
	\end{theoremB}
	
	The main tools involved in this paper are Rotation Theory and Forcing Theory for Surface Homeomorphisms. The first one we discuss on Section \ref{sec:preliminaries}, as well as some other basic dynamical lemmas. In Section \ref{sec:forcing} we discuss some particular results from Forcing Theory, the main new tools developed recently that allow us to obtain the main results of this paper. Section \ref{sec:thmA} and \ref{sec:thmB} we present the proofs of Theorem A and Theorem B, respectively.
	
	\section{Preliminaries}\label{sec:preliminaries}
	We will endow $\R^2$ with its usual orientation and with the Euclidean scalar product $\langle\cdot\;,\cdot\rangle$. We will write $||\cdot||$ and $\dist(\cdot\,,\cdot)$ for the associated norm and metric, respectively. We will consider the $n$-dimensional torus $\T^n$ as the quotient space $\R^n/\Z^n$ and we will denote by $\A=\T^1\times\R$ the (vertical) annulus. Also, we will denote the $2$-dimensional sphere by $\mathbb{S}^2$.

	Let $S$ be an oriented surface. 
	A subset $X$ of  $S$ is called an \textit{open disk} if it is homeomorphic to $D=\{z\in\R^2,\;||z||<1|\}$, a \textit{closed disk} if it is homeomorphic to $\overline{D}=\{z\in\R^2,\;||z||\leq1\}$, and it is called an \textit{annulus} if it is homeomorphic to $\T^1\times J$, where $J$ is a non-trivial interval of $\R$. If $J=[0,1]$ or $J=(0,1)$ we have a \textit{closed} and an \textit{open annulus}, respectively. 
	
	Let $\homeo(S)$ be the space of {homeomorphisms} defined on $S$, furnished with the $C^0$ topology. Given $f\in\homeo(S)$ and $z\in S$, the \textit{$\alpha$-limit} and \textit{$\omega$-limit sets of $z$} are, respectively,
	\[\alpha_f(z)=\bigcap_{k\leq 0}\overline{\left\{f^n(z)\;|\; n\leq k \right\}},\;\;\mbox{and}\;\;\omega_f(z)=\bigcap_{k\geq 0}\overline{\left\{f^n(z)\;|\; n\geq k \right\}}.\]

	We say that a point $z\in S$ is \textit{positively} or \textit{negatively recurrent} if $z\in\omega_f(z)$ or $z\in\alpha_f(z)$, respectively, and \textit{bi-recurrent} if it is both positively and negatively recurrent.
	
	\begin{lemma}\label{lemma:transitive}
		Let $f\in\homeo(S)$. Take any integer $q\geq1$ and let $g=f^q$. If there is a point {$z\in S$} such that $\omega_{f}(z)=S$ then \[\bigcup_{r=0}^{q-1}\omega_g(f^{r}(z))=S.\] 
	\end{lemma}
	\begin{proof}
		Let $w\in S$ be any point. As $\omega_f(z)=S$ then $w\in\omega_f(z)$. So,
		there is an increasing sequence $(n_k)_{k\in\N}$ of positive integers such that $n_k\nearrow+\infty$ and \[\lim_{k\to+\infty}f^{n_k}(z)=w.\]

		As $n_k\in\N$ for all $k\in\N$, there exist $m_k\in\N$ and $r_k\in\{0,1,\cdots,q-1\}$ such that $n_k=m_kq+r_k$. Then, we have \[f^{n_k}(z)=f^{m_kq+r_k}(z)=g^{m_k}(f^{r_k}(z)),\;\mbox{for each}\;k\in\N.\]
		
		{By the Pigeonhole Principle  there exists} a subsequence $(r_{k_i})_{i\in\N}$ of $(r_k)_{k\in\N}$ and  $r\in\{0,1,\cdots,q-1\}$ such that $r_{k_i}=r$ for all $i\in\N$. So, if we consider the subsequence  \[\left(g^{m_{k_i}}(f^{r_{k_i}}(z))\right)_{i\in\N}=\left(g^{m_{k_i}}(f^{r}(z))\right)_{i\in\N}\subset\left(g^{m_{k}}(f^{r_{k}}(z))\right)_{k\in\N},\] as a subsequence of a convergent sequence also converges to the same limit of it, we have \[\lim_{i\to+\infty}g^{m_{k_i}}(f^{r}(z))=\lim_{k\to+\infty}g^{m_{k}}(f^{r_{k}}(z))=\lim_{k\to+\infty}f^{n_k}(z)=w.\]
		
		Therefore, there exists some $r\in\{0,1,\cdots,q-1\}$ such that $w\in\omega_g(f^r(z))$, or equivalently, $w\in\bigcup_{r=0}^{q-1}\omega_g(f^{r}(z))$ as we wanted.
	\end{proof}
	
	We can also establish an analogous result for the $\alpha$-limit set:
	
	\begin{lemma}
		Let $f\in\homeo(S)$. Take any integer $q\geq1$ and let $g=f^q$. If there is a point $z\in\T^2$ such that $\alpha_{f}(z)=S$ then \[\bigcup_{r=0}^{q-1}\alpha_g(f^{-r}(z))=S.\] 
	\end{lemma}
	
	Recall that, if $f\in\homeo(S)$ is a topologically transitive map (or just a {\it{ transitive map}} for simplicity) on a compact surface then there is some $z\in S$ such that $\omega_f(z)=\alpha_f(z)=S$.
	Given $f\in\homeo(S)$ we say that a point $ z \in S $  is \textit{non-wandering for $f$}\label{nonwandering} if any neighborhood $ U \subset S $ of $ z $ is such that $ f^n (U) \cap U \neq \emptyset $ for some $ n \in \N $. Otherwise, the point is said \textit{wandering for $f$}. We represent the set of the non-wandering points for $f$ by $ \Omega (f) $. Moreover, we say that $f\in\homeo(S)$ is a non-wandering homeomorphism if every point $z\in S$ is non-wandering for $f$, that is $\Omega(f)=S$. \label{def:nonwandering}
	
	The following definition is about Birkhoff recurrence classes.
	
	\begin{definition}\label{def:birkhoff}
		Let $X$ be a metric space, $f\in\homeo(X)$,  
		\begin{enumerate}
			\item Given $z_1 , z_2$ two points of $X$, we say that there exists a \textit{Birkhoff connection from $z_1$ to $z_2$} if,
			for every neighborhood $W_1$ of $z_1$ and every neighborhood $W_2$ of $z_2$, there is {$n \geq 1$} such that $W_1\cap f^{-n}(W_2)\neq\emptyset$.
			\item A \textit{Birkhoff cycle} is a finite sequence $(z_i)_{i\in\Z/p\Z}$  of points in $X$ such that for every $i\in\Z/p\Z$ there is a Birkhoff connection from $z_i$ to $z_{i+1}$. 
			\item A point $z$ is said to be \textit{Birkhoff recurrent for f} if there is a Birkhoff cycle containing $z$. 
			\item We can define an equivalence relation in the set of Birkhoff recurrent points: say that \textit{$z_1$ is Birkhoff equivalent to $z_2$} if there exists a Birkhoff cycle containing both points. The equivalence classes will be called \textit{Birkhoff recurrence classes}.  
		\end{enumerate}
	\end{definition}
	Note  that every $\omega$-limit set or $\alpha$-limit set is contained in a single Birkhoff recurrence class. 
	
	Before finishing this subsection, let us fix some notations. Recall that, given $f\in\homeo(S)$, a \textit{lift} of $f$ to a covering space $\breve S$ of $S$ is a homeomorphism $\breve f:\breve S\to\breve S$ such that $f\circ \breve\pi=\breve\pi\circ\breve f$, where $\breve\pi:\breve S\to S$ is a covering projection map.

	\begin{remark}\label{remark:horseshoe}
		It follows from Definition \ref{def:horseshoe} that given $f\in\homeo(S)$ and a lift $\breve f\in\homeo(\breve S)$ of $f$, if $\breve f$ has a topological horseshoe then $f$ also has a topological horseshoe.
	\end{remark}

	Let $\breve S$ be the universal covering of $S$, that is a covering space that is simply connected, and $\breve\pi:\breve S\to S$ be a universal covering projection. A \textit{covering automorphism of $S$} is a homeomorphism $T:\breve S\to\breve S$ such that $\breve\pi \circ T=\breve\pi$. 
	The set of covering automorphism of the surface $S$ will be denoted by $\deck(S)$. 
	
	As is already known, we have that $\R^2$ is the universal covering space of $\T^2$ and  $\A$. Moreover $\A$ is also a covering space of $\T^2$. 
	\begin{notation*}\label{not:notacao} Let $\R^2$, $\T^2$ and $\A$ be as before, then:
		\begin{enumerate}
			\item A point in $\R^2$ will be denoted with a check mark: $\check z=(\check x,\check y)\in\R^2$. A point in $\A$ will be denoted with a hat mark: $\hat z=(\hat x,\hat y)\in\A$. And, finally, a point in $\T^2$ will be denoted without any mark: $z=(x,y)\in\T^2$;
			\item $\pr_1: \R^2\to\R$ and $\pr_2: \R^2\to\R$ denote the two canonical projections: $\pr_1(\check x,\check y)=\check x$ and $\pr_2(\check x,\check y)=\check y$. And we will also use $\pr_2$ to denote the second coordinate (vertical) projection in $\A$;
			\item $\check\pi:\R^2\to\T^2$ denotes the canonical universal covering projection of the $2$-torus: \[\check\pi((\check x,\check y))=(\check x+\Z,\check y+\Z).\]
			\item $\check\tau:\R^2\to\A$ denotes the canonical universal covering projection of the annulus: \[\check\tau((\check x,\check y))=(\check x+\Z, \check y).\]
			\item $\hat\pi:\A\to\T^2$ denotes the covering projection from the annulus to the $2$-torus, such that \[\check\pi=\hat\pi\circ\check\tau.\]
			\item A self-homeomorphism of $\R^2$ will be denoted with a check mark: $\check f\in\homeo(\R^2)$. For self-homeomorphisms of $\A$ we will  denote them with a hat mark: $\hat f\in\homeo(\A)$. And in turn, a self-homeomorphism of $\T^2$ will be denoted without any mark: $f\in\homeo(\T^2)$.	
		\end{enumerate}
	\end{notation*}
	
	\vspace{.15cm}Finally, we will endow $\T^2$ and $\A$ with the following metric:
	\[\begin{array}{l}
		d_{\T^2}(z',z)=\inf\{d(\check{z}',\check{z})\;|\;\check{z}'\in\check\pi^{-1}(z'),\check{z}\in\check\pi^{-1}(z)\},\;\mbox{and}\vspace{.2cm}\\
		d_{\A}(\hat z',\hat z)=\inf\{d(\check{z}',\check{z})\;|\;\check{z}'\in\check\tau^{-1}(\hat z'),\check{z}\in\check\tau^{-1}(\hat z)\},\;\mbox{respectively.}
	\end{array}\] 
	
	\subsection{Rotation theory for 2-torus homeomorphism}
	Given a surface $S$, we will denote by $\homeo_0(S)$ the subspace of $\homeo(S)$ of the homeomorphisms that are isotopic to identity.
	
	Let $f\in\homeo_0(\T^2)$ and $\check\pi:\R^2\to\T^2$ the canonical universal covering projection of $\T^2$. 
	Misiurewicz and Ziemian in \cite{rotationset} defined the rotation set of a lift $\check{f}\in\homeo(\R^2)$ of  $f$, denoted by $\rho(\check{f})$, as the set of all limit points of the following sequences
	\begin{equation*}
		\left(\frac{\check{f}^{n_k}(\check{z}_k)-\check{z}_k}{n_k}\right)_{k\geq 1},
	\end{equation*}
	where $(n_k)_{k\geq 1}$ is an increasing sequence of integers and $(\check{z}_k)_{k\geq 1}\subset\R^2$.
	
	If a point $z\in\T^2$ is such that the limit \[\lim_{n\to+\infty}\dfrac{\check f^n(\check z)-\check z}{n}\;\;\mbox{exists,}\] for any $\check z\in\check\pi^{-1}(z)$, then $\rho(z,\check f)$ is denoted as the limit above and it is called the rotation vector of $z$.
	
	We will say that $v\in\R^2$ is a \textit{rational vector}, if both coordinates of $v$ are rationals numbers.  
	Let $v=\left(\frac{r_1}{s},\frac{r_2}{s}\right)\in\rho(\check f)$ be a rational vector such that $r_1,r_2$ and $s$ are integers mutually coprime (that means $\gcd(r_1,r_2,s)=1$), $v$ is said to be \textit{realized by a periodic orbit}\label{def:periodicorb} if there exists $z\in\T^2$ such that 
	$\check f^s(\check z)=\check z+(r_1,r_2)$
	for any $\check z\in\check\pi^{-1}(z)$.

	We can also use invariant measures to average the spatial displacement, instead of averaging the displacement for a single point.
	Denote by $\mathcal{M}(f)$ the set of all $f$-invariant Borel probability measures on $\T^2$ and $\mathcal{M}_E(f)$ its subset of ergodic measures. {For each $\mu\in\mathcal{M}(f)$ and each lift $\check f$ of $f$, define the \textit{rotation number} \[\rho_{\mu}(\check f)=\int_{\T^2}\varphi\operatorname{d}\mu,\] where the \textit{displacement function} $\varphi:\T^2\to\R^2$ is defined by  $\varphi(z)=\check{f}(\check{z})-\check{z}$, for any $\check{z}\in{\check\pi}^{-1}(z)$.}
	
	The following proposition collects some results about the rotation set.
	
	\begin{proposition}[see \cite{rotationset}]\label{prop:rotationset}
		Let $f\in\homeo_0(\T^2)$ and $\check f\in\homeo(\R^2)$ be some lift of $f$. The following properties hold:
		\begin{enumerate}
			\item $\rho(\check{f})$ is a non-empty, closed, compact and convex set of $\R^2$;
			\item $\rho(\check f^q+p)=q\rho(\check f)+p$, where $q\in\Z$ and $p\in\Z^2$;
			\item If $\mu\in\mathcal M_E(f)$ is such that $\rho_{\mu}(\check f)=a$ then for $\mu$-almost every point $z\in\T^2$ and any $\check{z}\in{\check\pi}^{-1}(z)$, \[\lim_{n\to\infty}\dfrac{\check{f}^{n}(\check{z})-\check{z}}{n}=a;\]
			\item If $a\in\rho(\check f)$ is an extremal point (in the sense of convex set) then there exists $\mu\in\mathcal{M}_E(f)$ such that $\rho_{\mu}(\check f)=a$;
			\item $\rho(\check f)=\rho_{mes}(\check f)$.
		\end{enumerate}
	\end{proposition}
	
	Furthermore, under topological conjugacies, the rotation set has the following behavior:
	\begin{lemma}[Lemma 2.4 in \cite{kk08}]\label{lemma:kk08}
		Let $f, g\in\homeo(\T^2)$ and $A\in\operatorname{GL}(2,\Z)$. Suppose that $f\in\homeo_0(\T^2)$ and $g$ is isotopic to $f_A$, the linear automorphism induced by $A$ on $\T^2$. Let $\check f$ and $\check g$ be respective lifts of $f$ and $g$ to $\R^2$, then \[\rho(\check g\circ\check f\circ \check g^{-1})=A\rho(\check f).\] In particular, $\rho(A\check fA^{-1})=A\rho(\check f)$.
	\end{lemma}
	
	\begin{remark}\label{remark:kk08}
		The map $A\check{f}A^{-1}$ is a lift to $\R^2$ of a homeomorphism of $\T^2$ topologically conjugated to ${f}$, namely $f_A\circ f\circ f_A^{-1}$ where $f_A$ is the automorphism of $\T^2$ induced by $A$. Therefore,  if we wish to prove the existence of a property for $f$ that is invariant by conjugation, Lemma \ref{lemma:kk08} tell us that we can consider some simpler case: where the rotation set is the image of $\rho(\check f)$ by some element of $\operatorname{GL(2,\Z)}$.
	\end{remark}

	We will say that a non-degenerate line segment $l\subset\R^2$ is a \textit{non-degenerate line segment with rational slope} if there is some $\lambda\in\R\backslash\{0\}$ such that the direction of $l$ is $v=\lambda(v_1,v_2)$ where $(v_1,v_2)\in\Z\backslash\{(0,0)\}$.
	
	By a \textit{horizontal segment} and \textit{vertical segment} we mean a non-degenerate line segment whose its direction is given by $\pm(1,0)$ and $\pm(0,1)$, respectively. 
	
	The next corollary will be very useful in the sense of Remark \ref{remark:kk08}.
	
	\begin{corollary}\label{cor:kk08}
		If $\rho(\check{f})$ is a non-degenerate line segment of rational slope then there is a matrix $A\in\operatorname{GL}(2,\Z)$ such that $\rho(A\check{f}A^{-1})$ is a horizontal segment.
	\end{corollary} 
	\begin{proof}
		Indeed, if $\rho(\check f)$ has a rational slope then there are $v_1,v_2\in\Z$ coprimes and $\lambda\in\R\backslash\{0\}$ such that the direction of $\rho(\check f)$ is given by $v=\lambda\left(v_1,v_2\right)$.
		
		Then Bezout's Identity implies that there exist $x,y\in\Z$ such that $p_1x+p_2y=1$. Thus the matrix \[A=\begin{pmatrix}
			x & y\\
			-v_2 & v_1
		\end{pmatrix}\] is in $\operatorname{GL}(2,\Z)$ (in fact, $A\in\operatorname{SL}(2,\Z)$), and $Av=\begin{pmatrix}
			\lambda\\0
		\end{pmatrix}\vspace{.05cm}$. Therefore the direction of $A\rho(f)$ is given by $\dfrac{Av}{||Av||}=\pm(1,0)\vspace{.1cm}$, which is a horizontal segment, and by Lemma \ref{lemma:kk08}, we have that $\rho(A\check{f}A^{-1})=A\rho(f)$.
	\end{proof}
	
	Let $\check f\in\homeo(\R^2)$, we say that the \textit{deviations in the direction of some nonzero vector $v\in\R^2$ are uniformly bounded}, if there exists a real number $M>0$ such that \[\left|\pr_{v}(\check f^n(\check z)-\check z)\right|\leq M,\;\;\forall n\in\Z\;\;\mbox{and}\;\;\forall \check z\in\R^2,\] where $\pr_{v}(\cdot)=\langle\cdot,\frac{v}{||v||}\rangle$.
	
	\begin{definition}\label{def:annular}
		Let $f\in\homeo_0(\T^2)$.  We say that $f$ is \textit{annular} if there is an integer $q\geq 1$ and a lift $\check g\in\homeo(\R^2)$ of $g=f^q$ such that the deviations in the direction of some nonzero integer vector $v\in\Z^2$ are uniformly bounded.
	\end{definition}
	
	The next theorem, which we use often in this work, is due to D{\'a}valos, improving on a similar result for non-wandering homeomorphisms by Guelman-Koropecki-Tal (\cite{GKM}): 
	
	\begin{theorem}[Theorem A in \cite{davalos18}]\label{teo:davalos}
		If $\rho(\check f)$ is a non-degenerate line segment with rational slope containing rational points, then $f$ is annular in the direction perpendicular to $\rho(\check f)$.
	\end{theorem}
	
	As a consequence of this theorem we have the following:
	
	\begin{corollary}\label{cor:davalos}
		If $\rho(\check f)$ is a non-degenerate line segment with rational slope containing the origin then there is a lift $\check f$ of $f$ which has uniformly bounded deviations in the direction perpendicular to $\rho(\check f)$.
	\end{corollary}
	\begin{proof}
		Let $v=(v_1,v_2)\in\Z^2\backslash\{(0,0)\}$ be a multiple of the direction of $\rho(\check f)$ and let $v^{\perp}=(-v_2,v_1)$. 
		By Theorem \ref{teo:davalos} as $\rho(\check f)$ is a non-degenerate line segment with rational slope containing the origin then there is an integer $q\geq 1$ and a lift $\check g:\R^2\to\R^2$ of $g=f^q$ such that the deviations in the  direction perpendicular to $\rho(\check f)$ are uniformly bounded, that is \[\left|\pr_{v^{\perp}}(\check g^n(\check z)-\check z)\right|\leq M,\;\;\forall n\in\Z\;\;\mbox{and}\;\;\forall \check z\in\R^2,\]
		{ which, in particular, implies that $\rho(\check g)$ is a line segment in the direction of $v$.}

		The lift $\check g\in\homeo(\R^2)$ of $g=f^q$ can be considered as $\check g=\check f^q$. Otherwise there would be  $p\in\Z^2\backslash\{(0,0)\}$ such that $\check g=\check f^q+p$ has uniformly bounded deviations in the direction perpendicular to the rotation set. But  { as, by assumption, $(0,0)\in\rho(\check f)$, and since $\rho(\check g)=q\rho(\check f)+p$, we have that $p$ is in $\rho(\check g)$, which implies that in this case, we must have that $p\in\Z^2\cap\{\lambda v\;|\;\lambda\in\R\}$.}
		
		So, we have that $p=\lambda v$, for some $\lambda\in\R$. And, as $\pr_{v^{\perp}}(\lambda v)=0$, there is no loss of generality {supposing} that $\check g=\check f^q$.

		As for all $n\in\N$ there is $m\in\N$ and $r\in\{0,1,\cdots,q-1\}$ such that $n=mq+r$ and as $\{0,1,\cdots,q-1\}$ is finite we must have 
		$\left|\pr_{v^{\perp}}(\check f^r(\check w)-\check w)\right|\leq M_0$ for some constant $M_0>0$, for all $r\in\{0,1\cdots,q-1\}$. So \[\begin{array}{rl}
			\left|\pr_{v^{\perp}}(\check f^n(\check w)-\check w)\right| = &\hspace{-.2cm}\left|\pr_{v^{\perp}}(\check f^{mq}(\check f^r(\check w))-\check w )\right|\vspace{.1cm}\\
			\leq&\hspace{-.2cm}\left|\pr_{v^{\perp}}(\check g^m(\check f^r(\check w))-\check f^r(\check w))\right|+\left|\pr_{v^{\perp}}(\check f^r(\check w)-\check w)\right|\vspace{.1cm}\\
			\leq & M+M_0,\;\;\forall n\in\Z\;\;\mbox{and}\;\;\forall \check w\in\R^2.
		\end{array}\]
		
		Which means that $\check f$ has uniformly bounded deviations in the direction perpendicular to $\rho(\check f)$.
	\end{proof}
	
	\subsection{Rotation theory for annulus homeomorphism}\label{subsec:rotation_annulus}
	
	Let us now consider the annulus $\mathbb{A}=\T^1\times\R$,  $\check\tau:\R^2\to\mathbb{A}$ its canonical universal covering projection and $\hat f\in\homeo_0(\A)$. Let $\check f\in\homeo(\R^2)$ be a lift of $\hat f$ to $\R^2$.
	
	We define the set $\Ne^+(\hat f)$ as the set of all points $\hat z\in\A$ such its $\omega$-limit set is non-empty, $\Ne^-(\hat f)=\Ne^+(\hat f^{-1})$ as the set of all points $\hat z\in\A$ such its $\alpha$-limit set is non-empty and, finally $\Ne(\hat f)=\Ne^+(\hat f)\cup\Ne^-(\hat f)$.
	
	Let $\hat f\in\homeo_0(\A)$ and $\check{f}\in\homeo(\R^2)$ be a lift of $\hat f$ to $\R^2$. As defined by Le Roux in \cite{leroux13} and Conejeros \cite{conejeros18}, { a point $\hat z\in\Ne^+(\hat f)$ has a \textit{rotation number}, denoted by $\rot(\hat z,\hat f)$,}
	if for every compact set $\hat K\subset\mathbb{A}$ and every increasing
	sequence of integers $(n_k )_{k\geq 1}$ such that $\hat f^{n_k}(\hat z)\in \hat K$, the limit \[\lim_{k\to\infty}\dfrac{\pr_1(\check{f}^{n_k}(\check{z})-\check{z})}{n_k}\;\mbox{exists,}\]
	for any $\check{z}\in\check\tau^{-1}(\hat z)$.

	The following two results are due to Le Calvez and the second author (\cite{ferradura}):
	
	\begin{theorem}[Theorem A in \cite{ferradura}]\label{teo:A}
		Let $\hat f$ be a homeomorphism isotopic to identity on $\mathbb{A}$ and $\check{f}$ a lift of $\hat f$ for $\R^2$. Suppose that $\hat f$ has no topological horseshoe, then
		\begin{enumerate}
			\item each point $\hat z\in\Ne^+(\hat f)$ has a well-defined rotation number $\rot(\hat  z,\check{f})$;
			\item {for all $\hat z,\hat z'\in\Ne^+(\hat f)$} such that $\hat z'\in\omega_{\hat f}(\hat z)$, we have $\rot(\hat  z', \check{f})=\rot(\hat  z, \check{f})$;
			\item if $z\in\Ne^+(\hat f)\cap\Ne^+(\hat f^{-1})$ is non-wandering, then $\rot(\hat  z, \check f^{-1})=-\rot(\hat  z, \check f)$;
			\item the map $\rot_{\check f^{\pm}}:\Omega(\hat f)\cap\Ne(\hat f)\to\R$ is continuous, where \[\rot_{\check f^{\pm}}(\hat z)=\left\{\begin{array}{ll}
				\rot(\hat  z, \check f) & \mbox{if}\;\; \hat z\in\Omega(\hat f)\cap\Ne^+(\hat f)\\
				-\rot(\hat  z, \check f^{-1}) & \mbox{if}\;\; \hat z\in\Omega(\hat f)\cap\Ne^-(\hat f).\\
			\end{array}\right.\]
		\end{enumerate}
	\end{theorem}

	Let $\hat f\in\homeo_0(\A)$, we will denote $\hat f_{\sphere}$ as the continuous extension of $\hat f$ to the sphere obtained by adding the two ends $N$ (superior end) and $S$ (inferior end) of $\A$.
	
	\begin{proposition}[Proposition $D$ in (\cite{ferradura})]\label{prop:D}
		Let $\hat f$ be a homeomorphism of $\A$ isotopic to the identity, $\check f$ a lift of $\hat f$ to $\R^2$. We suppose that:
		\begin{enumerate}
			\item $\check{f}$ has at least two rotation numbers well defined and distinct;
			\item $N$ and $S$ belong to the same Birkhoff recurrence class of $\hat f_{\sphere}$.	 	
		\end{enumerate}
		Then $\hat f$ has a topological horseshoe.
	\end{proposition}
	
	\subsection{Singular oriented foliations}
	
	Let $S$ be an oriented surface. A \textit{path} is a continuous map $\gamma:J\to S$ where $J\subset\R$ is an interval. From now on, in absence of ambiguity, we will denoted the image of $\gamma$ also as $\gamma$ and call it by path. We will denote $\gamma^{-1}:-J\to S$ the path defined by $\gamma^{-1}(t)=\gamma(-t)$. If $X$ and $Y$ are two disjoint subset of $S$ we will say that a path $\gamma:[a,b]\to S$ \textit{joins $X$ to $Y$} if $\gamma(a)\in X$ and $\gamma(b)\in Y$. Moreover, if $\gamma:[a,b]\to S$ and $\gamma':[a',b']\to S$ are two paths such that $\gamma(b)=\gamma'(a')$ then we can concatenate the two paths and define the path $\gamma\gamma':[a,b+b'-a']\to S$
	\[\gamma\gamma'(t)=\begin{cases}
		\gamma(t) & \mbox{if}\;\; t\in[a,b];\\
		\gamma'(t-b+a') & \mbox{if}\;\; t\in[b,b+b'-a']\;.
	\end{cases}\]
	
	A path $\gamma:J\to S$ is \textit{proper} if $J$ is open and the preimage of every compact subset of $S$ is compact. If $S=\R^2$, a \textit{line} is an injective and proper path $\phi:J\to \R^2$, it inherits a natural orientation induced by the usual orientation on $\R$. We already know, by Jordan-Schoenflies Theorem, that the complement of a line $\phi$ has two connected components. We will denote $R(\phi)$ for the one that is on the right of $\phi$ and $L(\phi)$ for the other one that is on the left of it.

	A path $\gamma:\R\to S$ such that $\gamma(t+1)=\gamma(t)$ for every $t\in\R$ lifts a continuous map $\Gamma:\T^1\to S$. We will say that $\Gamma$ is a \textit{loop} and $\gamma$ is its \textit{natural lift}. When the map $\Gamma$ is also injective then $\Gamma$ will be called \textit{simple loop}. If $n\geq 1$, we denote $\Gamma^n$ the loop lifted by the path $t\mapsto \gamma(nt)$, and we will say that $\Gamma^n$ is a \textit{multiple} of the loop $\Gamma$.

	By a \textit{singular oriented foliation} on an oriented surface $S$ we mean a closed set $\sing(\F)$, called the \textit{set of singularities}, together with a topological oriented foliation $\F'$ on the complement of $\sing(\F)$, that we will denote by $\dom(\F):=S\backslash\sing(\F)$ and call it \textit{domain of $\F$}. 
	
	If $\sing(\F)$ is an empty set we will say that $\F$ is a non-singular oriented foliation. 
	
	By Whitney's Theorem, see \cite{w33} and \cite{w41}, any singular oriented foliation $\F$ can be embedded in a
	flow, which means that $\F$ is the set of oriented orbits of some topological flow $\Phi: S\times\R\to S$, where the singularities of $\F$ coincides with the set of the fixed points of $\Phi$. Thus we can define the $\omega$-limit and $\alpha$-limit sets of a leaf $\phi$ as follows: if $\phi$ is a leaf of $\F$ and $z\in\phi$ then \[\alpha(\phi)=\bigcap_{n\geq 0}\overline{\left\{\Phi(z,t)\;|\;t\leq n\right\}},\;\;\mbox{and}\;\;\omega(\phi)=\bigcap_{n\geq 0}\overline{\left\{\Phi(z,t)\;|\;t\geq n\right\}}.\]
	
	We will denote $\phi^-_z$ as the negative half-leaf and $\phi^+_z$ as the positive half-leaf from the point $z$, which means that if $\phi_z=\Phi(z,0)$ then $$\phi^-_z=\left\{\Phi(z,t);\;t< 0\right\},\;\mbox{and}\;\phi^+_z=\left\{\Phi(z,t);\;t> 0\right\}.$$
	
	If $\breve{S}$ is the universal covering space of $S$ and ${\breve\pi}:\breve{S}\to S$ a universal covering projection then $\F$ can be naturally lifted to a singular foliation $\breve{\F}$ of $\breve{S}$ such that $\dom(\breve{\F})={\breve\pi}^{-1}(\dom(\F))$. 
	
	We will denote $\tildedom(\F)$ as the universal covering space of $\dom(\F)$, $\tilde\pi:\tildedom(\F)\to\dom(\F)$ the universal covering projection and $\tildeF$ the non-singular oriented foliation of $\tildedom(\F)$ lifted from $\F|_{\dom(\F)}$.
	
	Let $\F$ be a singular oriented foliation on an oriented surface $S$. 
	
	\begin{definition}\label{def:transverse}
		We will say that a path $\gamma:J\to\dom(\F)$ is \textit{positively transverse} to $\F$ (in the whole text we will say just $\F$-transverse) if for every $t_0\in J$ there exist {a neighborhood $W$ of $\gamma(t_0)$ and an orientation preserving continuous map $h:W\to(0,1)^2$} that maps the restricted foliation $\F|_{W}$ onto the oriented foliation by vertical lines oriented downwards such that the first coordinate of the map $t\mapsto h(\gamma(t))$ is strictly increasing in a neighborhood of $t_0$.
	\end{definition}
	
	Note that an $\F$-transverse path $\gamma$ does not contain any singularity and each intersection of $\gamma$ with a leaf of $\F$ is topologically transverse and cross the leaf locally ``from right to left''. 
	
	{ We need to define an equivalence relation between $\F$-transverse paths. 
		
		\begin{definition}\label{def:equivalent}
			Let $\gamma:J\to\dom(\F)$ and $\gamma':J'\to\dom(\F)$ be two $\F$-transverse paths. We will say that $\gamma$ and $\gamma'$ are $\F$-\textit{equivalent} if they can be lifted to $\tildedom(\F)$ into paths $\tildegamma:J\to\tildedom(\F)$ and $\tildegamma':J'\to\tildedom(\F)$ that meet exactly the same leaves.
	\end{definition}}
	
	\begin{definition}\label{def:recurrentpath}
		An $\F$-transverse path $\gamma:\R\to S$ will be called \textit{$\F$-positively recurrent} if for every interval $J\subset \R$ and every $t\in\R$ there exists an interval $J_0\subset[t,+\infty)$ such that $\gamma|_{J_0}$ is equivalent to $\gamma|_J$. It will be called \textit{$\F$-negatively recurrent} if for every interval $J\subset \R$ and every $t\in\R$ there exists an interval $J_0\subset(-\infty,t]$ such that $\gamma|_{J_0}$ is equivalent to $\gamma|_J$. It is \textit{$\F$-bi-recurrent} if it is both $\F$-positively and $\F$-negatively recurrent. 
	\end{definition}
	
	\section{Forcing theory for surface homeomorphisms}\label{sec:forcing}
	
	Forcing theory for surface homeomorphisms is based on Brouwer theory and the famous Brouwer translation theorem. This theorem says that given  $\check f\in\homeo(\R^2)$ fixed point free, that is called \textit{Brouwer's homeomorphisms}, then for all $\check z\in\R^2$ there is a line $\check \phi$ such that $\check z\in\check \phi$ and \[\check f^{-1}(\check \phi)\subset R(\check \phi)\quad\mbox{and}\quad \check f(\check\phi)\subset L(\check \phi),\] this line is called \textit{Brouwer's line}.
	
	To relate forcing theory with Brouwer's theory we use a consequence of the Equivariant Foliation Theorem due to Le Calvez (see \cite{lecalvez05}). For do this we must introduce Maximal Isotopies.
	
	\subsection{Maximal Isotopy, Equivariant Foliation}
	
	Let $\I_f$ be the space of  identity isotopies of $f\in\homeo_0(f)$. That is, $\I_f$ is the set of continuous paths defined on $[0,1]$ joining the identity to $f$ in the space $\homeo(S)$, furnished with the $C^0$ topology. 
	
	If $I=(f_t)_{t\in[0,1]}\in\I_f$, we define the \textit{trajectory of the point $z\in S$ along $I$} as the path \[I(z):t\mapsto I(z)(t)=f_t(z).\]  In general, we can define by concatenation $I^{n}(z)=\prod_{0\leq m<n}I(f^m(z))$ for every integer $n\geq 1$. Furthermore, we can define 
	\[I^{\N}(z)=\prod_{m\geq 0}I(f^m(z)),\; I^{-\N}(z)=\prod_{m\leq 0}I(f^m(z)),\; I^{\Z}(x)=\prod_{m\in\Z}I(f^m(z)),\]
	as the future, past and whole trajectory of a point $z\in S$ along $I$, respectively.

	\begin{definition}
		Given $f\in\homeo_0(S)$ and $I\in\I_f$, we define the \textit{fixed point set of $I$} as the set $\fix(I)=\bigcap_{t\in[0,1]}\fix(f_t).$ If $z\in\fix(I)$ then $z$ is called \textit{fixed point for the isotopy $I$}.
	\end{definition}
	
	The complement set of $\fix(I)$ in $S$, $S\backslash\fix(I)$ will be called domain of $I$ and will be denoted by $\dom(I)$, that is itself a surface. 
	
	We can define a partial order on $\I_f$ as follows: given $I'$ and $I$ two elements of $\I_f$, we say that $I'\preceq I$ if 
	\begin{enumerate}
		\item $\fix(I')\subset \fix(I)$ and
		\item $I$ is homotopic to $I'$ relative to $\fix(I')$.
	\end{enumerate}
	
	Thus, we say that $I$ is a maximal isotopy of $f$, if it is maximal for the partial order above.

	An important result regarding existence of maximal isotopy is due B{\'e}guin, Crovisier and Le Roux in \cite{bclr16}:
	
	\begin{theorem}[Maximal Isotopy, see Corollary 1.3 in \cite{bclr16}]\label{teo:bclr}
		Let $f:S\to S$ be a homeomorphism isotopic to identity, then for every $I'\in\mathcal{I}_f$, there exists $I\in\mathcal{I}_f$ such that $I'\preceq I$ and $I$ is maximal for the partial order above.
	\end{theorem}

	An isotopy $I$ is maximal if and only if, for every $z\in\fix(f) \setminus \fix(I)$, the closed curve $I(z)$ is not contractible in $\dom(I)$. So, {if $\tilde I=(\tilde f_t)_{t\in[0,1]}$ is the identity isotopy that lifts $I|_{\dom(I)}$ to the universal covering space $\tildedom(I)$ of $\dom(I)$ then $\tilde f_1$ has no fixed point, which implies that each connected component of $\tilde S$ is homeomorphic to the plane}. And therefore, as a consequence of the Equivariant Foliation Theorem, due to Le Calvez (see \cite{lecalvez05}), we have:
	
	\begin{corollary}\label{cor:equivariant}
		Let $f:S\to S$ be a {homeomorphism} isotopic to identity of a surface $S$, $I=(f_t)_{t\in[0,1]}$ a maximal identity isotopy on $S$ {such $f_1:=f$} then there exists a singular oriented foliation $\F$ on $S$ such that $\sing(\F)=\fix(I)$ and for all point $z\in \dom(I)$ there is a positively $\F$-transverse path $\gamma$ that joins $z$ to $f(z)$ and is homotopic to $I(z)$ on $\dom(I)$, with the endpoints fixed. 
	\end{corollary}
	
	The foliation $\F$ satisfying the above conclusion is said to be \textit{transverse to $I$}.

	\subsection{Transverse Trajectories and $\F$-transverse intersection}
	
	{ Note that the path $\gamma$, given by Corollary \ref{cor:equivariant}, is not uniquely defined. }
	We will write $I_{\F}(z)$ for the set of paths $\gamma:[a,b]\to\dom(I)$ defined on a compact interval $[a,b]$ such that $\gamma(a)=z,\, \gamma(b)=f(z)$, $\gamma$  is $\F$-transverse and $\gamma$ is homotopic to $I(z)$ in $\dom(I)$, relative to the endpoints. We will also abuse the notation  and use $I_{\F}(z)$ for any representative in this class and call it the \textit{transverse trajectory of $z$}.  As before, we can define by concatenation $I^{n}_{\F}(z)=\prod_{0\leq m<n}I_{\F}(f^m(z))$ for every integer $n\geq 1$. Similarly we define 
	$$I^{\N}_{\F}(z)=\prod_{m\geq 0}I_{\F}(f^m(z)),\; I^{-\N}_{\F}(z)=\prod_{m\leq 0}I_{\F}(f^m(z)),\; I^{\Z}_{\F}(x)=\prod_{m\in\Z}I_{\F}(f^m(z)),$$
	and the last one will be called the \textit{whole transverse trajectory of the point $z\in S$}.
	
	A main definition in forcing theory for surface homeomorphisms is about a kind of a topological intersection between transverse path, that we call \textit{$\F$-transverse intersection}. In order to express it, we need the concept of relative order for lines on $\R^2$:
	
\begin{definition}\label{def:acimaabaixo}
	Given three disjoint lines $\phi,\, \phi_1,\, \phi_2:\R\to\R^2$, we will say that $\phi_2$ \textit{is above} $\phi_1$ \textit{relative to} $\phi$ (and that $\phi_1$ \textit{is below} $\phi_2$\textit{ relative to} $\phi$) if none of the lines separates the two others and if, for every pair of disjoint paths $\gamma_1$ and $\gamma_2$ joining $z_1=\phi(t_1)$ to $z'_1\in\phi_1$ and $z_2=\phi(t_2)$ to $z'_2\in\phi_2$, respectively, such that the paths do not meet the lines but at their endpoints, we have that $t_1<t_2$. See the Figure \ref{fig:acimaabaixo}.
\end{definition}

\begin{figure}[!h]
	\centering
	\def\svgwidth{5cm}
	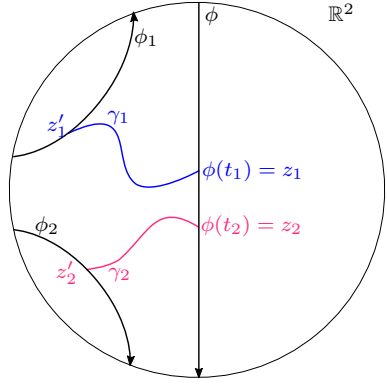
	\caption{$\phi_2$ is above $\phi_1$ relative to $\phi$.}
	\label{fig:acimaabaixo}
\end{figure}

\begin{definition}\label{def:Fintersection}
	Let $\gamma_1: J_1\to\dom(\F)$ and $\gamma_2: J_2\to\dom(\F)$ be two $\F$-transverse paths defined on intervals $J_1, J_2\subset\R$. Suppose that there exist $t_1\in J_1$ and $t_2\in J_2$ such that $\gamma_1(t_1)=\gamma_2(t_2)$. We will say that $\gamma_1$ and $\gamma_2$ have an {\textit{$\F$-transverse intersection}} at $\gamma_1(t_1)=\gamma_2(t_2)$ if there exist $a_1,b_1\in J_1$ and $a_2,b_2\in J_2$ satisfying $a_1<t_1<b_1$ and $a_2<t_2<b_2$ respectively, such that if $\tildegamma_1,\tildegamma_2$ are lifts of $\gamma_1,\gamma_2$ to $\tildedom(\F)$ respectively, verifying $\tildegamma_1(t_1)=\tildegamma_2(t_2)$, then
	\begin{enumerate}
		\item $\phi_{\tildegamma_1(a_1)}\subset L(\phi_{\tildegamma_2(a_2)}),\;\phi_{\tildegamma_2(a_2)}\subset L(\phi_{\tildegamma_1(a_1)})$; 
		\item $\phi_{\tildegamma_1(b_1)}\subset R(\phi_{\tildegamma_2(b_2)}),\;\phi_{\tildegamma_2(f_2)}\subset R(\phi_{\tildegamma_1(b_1)})$;
		\item $\phi_{\tildegamma_2(a_2)}$ is below $\phi_{\tildegamma_1(a_1)}$ relative to $\phi_{\tildegamma_1(t_1)}$ and $\phi_{\tildegamma_2(b_2)}$ is above $\phi_{\tildegamma_1(b_1)}$ relative to $\phi_{\tildegamma_1(t_1)}$.
	\end{enumerate}
\end{definition}
	
	\begin{figure}[!h]
		\centering
		\def\svgwidth{5cm}
		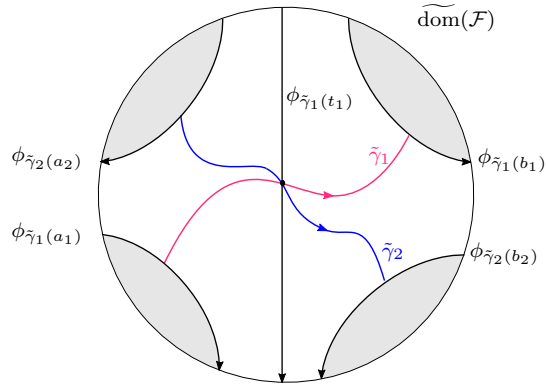
		\caption{$\gamma_1$ and $\gamma_2$ have an $\F$-transverse intersection.}
		\label{fig:intersecao}
	\end{figure}

	Definition \ref{def:Fintersection} means that any {path} $\tilde\alpha_1$ that joins the leaves $\phi_{\tilde\gamma_1(a_1)}$ and $\phi_{\tilde\gamma_1(b_1)}$ and any path $\tilde\alpha_2$ joining $\phi_{\tilde\gamma_2(a_2)}$ and $\phi_{\tilde\gamma_2(b_2)}$ must to intersect.
	
	When $\gamma_1=\gamma_2$ we will say that $\gamma_1$ has an \textit{$\F$-transverse self-intersection}. Namely, an $\F$-transverse path {$\gamma\in\dom(\F)$} has an $\F$-transverse self-intersection if for every lift $\tildegamma$ to the universal covering space $\tildedom(\F)$ of $\dom(\F)$, there exists a non trivial covering automorphism $T\in\deck(\dom(\F))$ such that $\tildegamma$ and $T(\tildegamma)$ have an $\tildeF$-transverse intersection.
	
	The statement below is a fundamental theorem of  \cite{ferradura}, which is a purely topological criterion for the existence of a topological horseshoe stated in terms of transverse trajectories.

	\begin{theorem}[Topological Horseshoe, Theorem N, \cite{ferradura}]\label{teo:horseshoe}
		Let $S$ be an oriented surface, $f$ a homeomorphism isotopic to identity on $S$, $I$ a maximal isotopy of $f$ and $\F$ a transverse foliation to $I$. If there exists a point $z\in\dom(I)$ and an integer $q\geq 1$ such that the transverse trajectory $I^q_{\F}(z)$ has an $\F$-transverse self-intersection at $I^q_{\F}(z)(s)=I^q_{\F}(z)(t)$, where $s<t$, then $f$ has a topological horseshoe.
	\end{theorem}

	\subsection{Other results}
	In this subsection, {we present several results} that will be useful for the proofs of our results. They may be found in \cite{forcing} and \cite{ferradura}.

	An $\F$-transverse loop $\Gamma$ is called \textit{prime} if there is no $\F$-transverse loop $\Gamma'$ and an integer $n\geq 2$ such that $\Gamma$ is $\F$-equivalent to $\Gamma^{'n}$. Note that if $\Gamma$ is a simple loop then it is prime.
	
	\begin{proposition}[Proposition 1 in \cite{forcing}]\label{prop:proposition1}
		Let $\F$ be a  singular oriented foliation on a surface and $(\Gamma_i)_{1\leq i \leq m}$ a family of prime
		$\F$-transverse loops that are not pairwise equivalent. We suppose that the leaves met by the loops $\Gamma_i$ are never closed and that there exists an integer $N$ such that no loop $\Gamma_i$ meets a leaf more than $N$ times.
		Then, for every $i \in\{1,\cdots,m\}$, there exists an $\F$-transverse loop $\Gamma'_i$ equivalent to $\Gamma_i$ such that:
		\begin{enumerate}
			\item $\Gamma'_i$ and $\Gamma'_j$ do not intersect if $\Gamma_i$ and $\Gamma_j$ have no $\F$-transverse intersection;
			\item $\Gamma'_i$ is simple if $\Gamma_i$ has no $\F$-transverse self-intersection.
		\end{enumerate}
	\end{proposition}
	
	\begin{remark}\label{prop:1}
		If $\F$ is a singular oriented foliation on an oriented surface $S$ that has genus $0$ then, by the Jordan Curve Theorem, an $\F$-transverse simple loop meats a leaf of $\F$ just one time and, consequently, the leaves met by it are never closed. 
	\end{remark} 
	
	The following result is Proposition $2$ in \cite{forcing} and it is an adapted version of the Poincar{\'e}-Bendixson Theorem.

	\begin{proposition}\label{prop:proposition2}
		Let $\F$ be a singular oriented foliation on $\mathbb{S}^2$ and $\gamma:\R\to\mathbb{S}^2$ an $\F$-bi-recurrent transverse path. The following properties are equivalent:
		\begin{enumerate}
			\item $\gamma$ has no $\F$-transverse self-intersection;
			\item There exists an $\F$-transverse simple loop $\Gamma_0$ such that $\gamma$ is equivalent to the natural lift $\gamma_0$ of $\Gamma_0$, and
			\item The set $U=\bigcup_{t\in\R}\phi_{\gamma(t)}$ is an open annulus.
		\end{enumerate}
	\end{proposition}
	As a scholium of the proof of this proposition we have that an $\F$-transverse loop $\Gamma$ with no $\F$-transverse self-intersection is equivalent to a multiple of an $\F$-transverse simple loop $\Gamma_0$.

	\begin{definition}[Section 4.1 in \cite{ferradura}]\label{def:draws}
		Let $\F$ be a singular oriented foliation on $\mathbb{S}^2$.	Let $\gamma: J\to\dom(\F)$ be an $\F$-transverse path and $\Gamma_0:\T^1\to\dom(\F)$ be an $\F$-transverse loop. We will say that $\gamma$ \textit{draws} $\Gamma_0$ if there exist $a<b$ in $J$ and $t\in\R$ such that $\gamma|_{[a,b]}$ is $\F$-equivalent to $\gamma_0|_{[t,t+1]}$, where $\gamma_0$ is a natural lift of $\Gamma_0$.
	\end{definition}
	
	Let $\gamma:J\to\dom(\F)$ be an $\F$-transverse path that draws a simple loop $\Gamma_0$ and let $U_{0}$ be the open annulus of leaves that are crossed by $\Gamma_0$. We have the set \[J^{\gamma}_{\Gamma_0}=\{t\in J\,|\,\gamma(t)\in U_{{0}}\}.\] 
	
	A connected component {$J_0\subset J^{\gamma}_{\Gamma_0}$} is a \textit{drawing component of $\gamma$} if $\gamma|_{J_0}$ draws $\Gamma_0$. A \textit{crossing component of $\gamma$} is a connected component {$J_1\subset J^{\gamma}_{\Gamma_0}$} such that both ends $a$ and $b$ of $J_1$ are in $J$ and $\gamma(a)$ and $\gamma(b)$ are in different components of $\mathbb{S}^2\backslash {\Gamma_0}$. In this case, we will say that $\gamma|_{J_1}$ \textit{crosses} $\Gamma_0$.

	The next result is Proposition $24$ in section $4.2$ of \cite{ferradura}. 
	
	\begin{proposition}\label{prop:20}
		Let $\F$ be a singular oriented foliation on sphere $\mathbb{S}^2$ and suppose that $\gamma:J\to\dom(\F)$ is an $\F$-transverse path with no $\F$-transverse self-intersection, then
		\begin{enumerate}
			\item If $\gamma$ draws an $\F$-transverse simple loop $\Gamma_0$, there exists a unique drawing component of  $J^{\gamma}_{\Gamma_0}$;
			\item If $\gamma$ draws and crosses an $\F$-transverse simple loop $\Gamma_0$, there exists a unique crossing component of $J^{\gamma}_{\Gamma_0}$ and it coincides with the drawing component;
			\item if $\gamma$ draws an $\F$-transverse simple loop $\Gamma_0$ and does not cross it, the drawing component of $J^{\gamma}_{\Gamma_0}$ contains a neighborhood of at least one end of $J$.
		\end{enumerate}
	\end{proposition}
	
	
	{	\begin{corollary}\label{cor:prop20}
			Let $\F$ be a singular oriented foliation on sphere $\mathbb{S}^2$. If $\gamma:J=(a,b)\to\dom(\F)$ is an $\F$-transverse path with no $\F$-transverse self-intersection that draws an $\F$-transverse simple loop $\Gamma_0$ then there is no $t_1,t_2\in J$ that are separated by the drawing component of $\gamma$ and such that $\gamma(t_1),\gamma(t_2)$ belongs to the same connected component of $\S^2\setminus U_0$, where $U_0$ is the open annulus of the leaves crossed by $\Gamma_0$.	
		\end{corollary}
		\begin{proof}
			Let $J_0=(c,d)\subset J$ be the unique drawing component of $\gamma$ given by item (1) of Proposition \ref{prop:20}. {If $J_0=J$ then the conclusion is obvious.} Otherwise, suppose, for a contradiction, that there exist $t_1,t_2\in J,\, t_1<c<d<t_2,$ such that $\gamma(t_1)$ and $\gamma(t_2)$ are in the same connected component $X$ of $\S^2\setminus U_0$. Since {$a<c<d<b$}, by item (3) of Proposition \ref{prop:20} we have that $\gamma$ crosses $\Gamma_0$, and then, by item (2) of the same proposition, we get that $J_0$ is also the unique crossing component. But $\gamma(c)$ and $\gamma(d)$ are in different connected components of $\mathbb{S}^2\backslash {\Gamma_0}$, while $\gamma(t_1)$ and $\gamma(t_2)$ are in the same connected component $X$ of $\mathbb{S}^2\backslash {\Gamma_0}$. If $\gamma(d)$ is not in $X$, there exists some connected component  $(e,f)\subset J^{\gamma}_{\Gamma_0}$, with {$d\le e<f\le t_2$} such that $\gamma(e)\not\in X$ and $\gamma(f)\in X$, a contradiction because this would imply the existence of a second crossing component. If, instead, $\gamma(c)$ is not in $X$, then one gets in the same way that there exists another crossing component $(e',f')$, with $t_1\le e'<f'\le c$, again a contradiction.
	\end{proof}}	

	For the next results let  $f\in\homeo_0(S)$, $I$ be a maximal isotopy for $f$ and $\F$ be a transverse foliation to $I$ given by Corollary \ref{cor:equivariant}.

	\begin{lemma}\label{lemma:17}
		For every $z\in\dom(I)$, and every integer $n\geq 1$:
		\begin{enumerate}
			\item There exists a neighborhood $W$ of $z$ such that, for every $z',z''\in W$, the path $I^n_{\F}(z')$ is $\F$-equivalent to a subpath of $I^{n+2}_{\F}(f^{-1}(z''))$. 
			\item For every {$z'\in\omega_f(z)\cap\dom(I)$ (respectively $z'\in\alpha_f(z)\cap\dom(I)$)} and every $m\geq 0$, the path $I^{n}_{\F}(z')$ is $\F$-equivalent to a subpath of $I^{\N}_{\F}(f^m(z))$ (respectively $I^{-\N}_{\F}(f^{-m}(z))$).
		\end{enumerate}
	\end{lemma}
	
	{\begin{remark}
			The first item of Lemma \ref{lemma:17} is part of Lemma 17 in \cite{forcing}. And the second item follows immediately from the first one. Just observe that for every $m\geq 0$ there is some integer $m_k>m$ such that $f^{m_k}(z)\in W$, where $W$ is the neighborhood given by item $(1)$.
	\end{remark}}
	
	It follows from the Lemma \ref{lemma:17} and from Proposition \ref{prop:20} the consequence below:
	
	\begin{corollary}[Proposition $27$ in \cite{ferradura}]\label{cor:prop23}
		Let $f$ be an orientation preserving homeomorphism of $\mathbb{S}^2$ with no topological horseshoe. Let $I$ be a maximal isotopy of $f$ and $\F$ be a  transverse foliation to $I$. If $z\in\dom(I)$ is a non-wandering point of $f$, then:
		\begin{enumerate}
			\item Either $I^{\Z}_{\F}(z)$ never meets a leaf twice;
			\item Or $I^{\Z}_{\F}(z)$ is equivalent to an $\F$-transverse simple loop $\Gamma_0$.
		\end{enumerate}
	\end{corollary}
	{\begin{remark}\label{remark:scholium}
			As a scholium of this corollary we have that if we change the hypothesis that $f$ has no topological horseshoe to the hypothesis that $I_{\F}^{\Z}(z)$ has no $\F$-transverse self-intersection, where $z\in\dom(I)$ is a non-wandering point of $f$  then the conclusion of the corollary is also true: either item (\textit{1.}) or item (\textit{2.}) hold.
	\end{remark}}
	
	It is worth noting that the first case happens only when $\alpha(z)$ and $\omega(z)$ are included in $\fix(I)$, see the remark after Proposition $27$ in \cite{ferradura}.

	{\begin{remark}
			Proposition \ref{prop:20} and Corollary \ref{cor:prop20} deal with foliations of $\S^2$ and Corollary \ref{cor:prop23} deals with homeomorphisms of $\S^2$, but in all three cases the results can be extended for foliations of $\R^2$ and homeomorphisms of $\R^2$, respectively, by just taking a one point compactification and extending the foliations with a  singularity at the added point, {for Proposition} \ref{prop:20} and Corollary \ref{cor:prop20}, and by extending the isotopy by fixing the added point in Corollary \ref{cor:prop23}.
	\end{remark}}
	
	\section {Proof of Theorem A}\label{sec:thmA}
	
	Our {first} goal is to show the existence of topological horseshoes for transitive homeomorphisms of $\T^2$ with a specific type of  rotation set.
	\begin{proposition}\label{prop:rational-rotationset}
		Let $f\in\homeo(\T^2)$ be topologically transitive with a non-empty fixed point set. If there is an integer $q\geq1$ such that the power $g=f^q$ of $f$ is isotopic to identity and the rotation set of a lift $\check g\in\homeo(\R^2)$ of $g$ is a non-degenerate line segment with a rational slope then $f$ has a topological horseshoe.
	\end{proposition}
	
	Before proving this proposition, let's prove some useful results.

	\begin{lemma}\label{lemma:omega-limit}
		Let $h\in\homeo(\T^2)$. Suppose that $h$ admits a lift $\hat h\in\homeo(\A)$ to the annulus. Moreover, suppose that
		\begin{enumerate}
			\item The map $\hat h$ commutes with the integer vertical translations, that is, 
			\begin{equation*}
				\hat h(\hat w + (0,p))=\hat h(\hat w)+(0,p),\;\forall p\in\Z\;\;\mbox{and}\;\;\hat w\in\A;
			\end{equation*}
			\item There is a real number $M>0$ such that for all $\hat w\in\A$ we have \begin{equation*}
				\left|\pr_2(\hat h^n(\hat w)-\hat w)\right|\leq M,\;\forall n\in\Z. 
			\end{equation*}
		\end{enumerate} 
		Then, if $w,w'\in\T^2$ are such that $w\in \omega_{h}(w')$, then for every $\hat w\in \hat\pi^{-1}(w)$ there exists some $\hat w'\in\hat\pi^{-1}(w')$ such that $\hat w\in \omega_{\hat h}(\hat w')$.
	\end{lemma}
	\begin{proof}
		Let $w,w'\in\T^2$ be as in the hypothesis. Choose some $\hat{w}\in \hat\pi^{-1}(w)$, then there is a fundamental domain $\hat D=\T^1\times[d,d+1)$ in $\A$ that contains $\hat w$. Moreover, for all $\hat u\in\hat D$ we have
		{	\begin{equation}\label{eq:w}
				d\leq\pr_2(\hat u)< d+1. 
		\end{equation} }
		
		As $w\in\omega_h(w')$ 	we may assume, possibly by taking a subsequence, that 
		\[
		d_{\T^2}(h^{m_k}(w'),w)<1/k\,\,,\;\mbox{for all $k\in \N$}.
		\]
		
		As $\hat{D}$ is a fundamental domain, the above property and assumption {(1)} imply that there exist {$\hat w'\in\hat\pi^{-1}(w')\cap\hat D$} and a sequence of integers $(p_k)_{k\in\N}\subset\Z$ such that 
		{	\begin{equation}\label{eq:ww'}
				d_{\A}(\hat h^{m_k}(\hat w')+(0,p_k),\hat w)<1/k\,\,,\;\forall k\in\N. 
		\end{equation}}	
		
		So, {it follows} from  (\ref{eq:w}) and (\ref{eq:ww'}) that 
		{	\begin{equation}\label{eq:pr21}
				d-\frac{1}{k}<\pr_2(\hat h^{m_k}(\hat w'))+p_k<d+1+\frac{1}{k}, \;\forall k\in\N. 
		\end{equation}}
		
		Furthermore, assumption (2) and inequality (\ref{eq:w}) applied to this $\hat w'\in\hat\pi^{-1}(w')\cap\hat D$ give us that	
		{	\begin{equation}\label{eq:pr22}
				d-M<\pr_2(\hat h^{m_k}(\hat w'))<d+1+M.
		\end{equation}}
		
		Finally,  (\ref{eq:pr21}) and (\ref{eq:pr22}) imply, for all $k\in\N$, that	
		\[\begin{array}{rl}
			-d-1-M+d-\dfrac{1}{k}<p_k<-d+M+d+1+\dfrac{1}{k}\Rightarrow&|p_k|<M+1+\dfrac{1}{k}<M+2
		\end{array}\]
		
		{\noindent	because we have that $k\in\N$ and then {$1/k\leq1$}.}
		
		This means that the discrete set $\{p_k\;|\;{k\in\N}\}\subset\Z$ is finite. Indeed the cardinality of $\{p_k\;|\;k\in\N\}$, namely $\#\{p_k\;|\;k\in\N\}$ is less than or equal to $2\lfloor M+2\rfloor+1$, where $\lfloor x\rfloor$ is the largest integer less than or equal to the positive real number $x$.
		
		Then,  by the Pigeonhole Principle, there is some fixed $p_0$ in $\Z$ and a subsequence $(p_{k_{i}})_{i\in\N}$ of $(p_{k})_{k\in\N}$ such that $p_{k_{i}} \equiv p_0$, for all $i\in\N$.
		
		So,\[d_{\A}(\hat h^{m_{k_i}}(\hat w'+(0,p_0)),\hat w)<1/k_i,\;\forall i\in\N\] and this imply that $\hat w\in\omega_h(\hat{w}'+(0,p_0))$.
	\end{proof}

	As a Corollary of Lemma \ref{lemma:transitive} and Lemma \ref{lemma:omega-limit} we have: 
	
	\begin{corollary}\label{cor:nonwandering}
		Let $f\in\homeo(\T^2)$ be topologically transitive. If there is a power $h=f^q$, $q\geq 1$, of $f$ that admits a lift $\hat h\in\homeo(\A)$ satisfying the assumptions $(1.)$ and $(2.)$ of Lemma \ref{lemma:omega-limit} then $\hat h$ is non-wandering.
	\end{corollary}
	\begin{proof}
		Let $\hat{w}\in\A$ be any point {and $z\in\T^2$ such that $\alpha_f(z)=\omega_f(z)=\T^2$}. Then, by Lemma \ref{lemma:transitive}, there exist some $r\in\{0,\cdots,q-1\}$ such that \[\hat\pi(\hat w)\in\omega_{h}(f^r( z)).\] 
		
		By Lemma \ref{lemma:omega-limit} we know that there is some lift $\hat z_r\in\hat\pi^{-1}(f^r(z))$ such that $\hat w\in\omega_{\hat h}(\hat z_r)$. Therefore $\hat w\in\A$ is non-wandering for $\hat h$.

		As $\hat w\in\A$ is any point, we conclude that $\hat h$ is non-wandering.
	\end{proof}

	\begin{proof}[Proof of Proposition \ref{prop:rational-rotationset}]
		Suppose, by contradiction, that $f$ does not have a topological horseshoe. Then no power of $f$ has a topological horseshoe.
		
		Let $q\geq 1$ be an integer such that the power $g=f^q$ of $f$ is isotopic to identity and the rotation set of a lift $\check g\in\homeo(\R^2)$ of $g$, is a non-degenerate line segment with a rational slope. 
		
		First of all, we will prove that $g=f^q$ is topologically conjugate to a homeomorphism $h\in\homeo(\T^2)$ that has a lift $\check h\in\homeo(\R^2)$ such that $\rho(\check h)$ is a horizontal segment containing the origin which is realized by periodic orbit.
		
		Let $z_0\in\fix(f)$, then $z_0\in\fix(g)$, and as $g$ is isotopic to identity, there is some $(p_1,p_2)\in\Z^2$ such that $\check g(\check z_0)=\check z_0+(p_1,p_2)$, {for any $\check z_0\in\check\pi^{-1}(z_0)$}. There is no loss of generality in supposing that $(p_1,p_2)=(0,0)$, because if it is not, then we can change the lift $\check g$ for the lift $\check g'=\check g-(p_1,p_2)$ which satisfies that $\check g'(\check z_0)=\check z_0$ and $\rho(\check g')=\rho(\check g)-(p_1,p_2)$. Thus $(0,0)\in\rho(\check g)$ is realized by $z_0\in\T^2$.
		
		As, by assumption $\rho(\check g)$ is a non-degenerate line segment with a rational slope then, by Corollary \ref{cor:kk08}, there is a matrix $A\in\operatorname{GL(2,\Z)}$ such that $\rho(A\check gA^{-1})=A\rho(\check g)$ is a horizontal segment and if $g_A$ is the homeomorphism of $\T^2$ induced by $A$ then we have that $g_A(z_0)\in\T^2$ is fixed by $g_Agg_A^{-1}$ and moreover
		
		\begin{equation*}
			\begin{array}{rl}
				\rho(g_A(z_0),A\check gA^{-1})=&\lim_{n\to+\infty}\dfrac{(A\check gA^{-1})^n(A\check z_0)-A\check z_0}{n}\vspace{.1cm}\\
				=&A\left(\lim_{n\to+\infty}\dfrac{\check g^n(\check z_0)-\check z_0}{n}\right)\vspace{.1cm}\\
				=&A\begin{pmatrix}
					0\\0
				\end{pmatrix}\vspace{.1cm}\\
				=&\begin{pmatrix}
					0\\0
				\end{pmatrix},
			\end{array}
		\end{equation*}
		for any $A\check z_0\in\check\pi^{-1}(g_A(z_0))$.
		
		Thus, under a topological conjugacy,  we can assume that $g=f^q$ is a homeomorphism of $\T^2$ isotopic to identity such  the rotation set of some lift $\check g\in\homeo(\R^2)$ of $g$ is a horizontal segment containing $(0,0)$ and such that $(0,0)\in\rho(\check g)$ is realized by $z_0\in\fix(g)$.
		
		By Corollary \ref{cor:davalos}, $\check g$ has uniformly bounded deviations in the perpendicular direction of $\rho(\check g)$, so there is a real number $M>0$ such that 
		{	\begin{equation}\label{eq:deviations}
				\left|\pr_2(\check g^n(\check w)-\check w)\right|\leq M,\;\;\forall n\in\Z\;\;\mbox{and}\;\;\forall\check w\in\R^2. 
		\end{equation}}
		
		Furthermore, as $\rho(\check g)$ is a non-degenerate line segment containing the origin, let $a\in\rho(\check g)\backslash\{(0,0)\}$ be an extremal point. Then we know ({items (3) and (4) of Proposition \ref{prop:rotationset}})  there is some point $z_a\in\T^2\backslash\{z_0\}$ such that $\rho(z_a,\check g)=a\neq(0,0)$.
		
		By assumption, we have that $f\in\homeo(\T^2)$ is topologically transitive. So, there is a point $z\in\T^2$ such that $\omega_f(z)=\T^2$. Then, by Lemma \ref{lemma:transitive}, we have 
		\begin{equation*}
			\bigcup_{r=0}^{q-1}\omega_g(f^r(z))=\T^2.
		\end{equation*}
		
		{Now, let $I\in\I_g$ be an identity isotopy of $g$ such that $z_0\in\fix(I)$, then $I$ is lifted to an identity isotopy $\check I$ of $\check g$. 
			
			The quotient space of $\R^2$ by a unit horizontal translation is homeomorphic to the annulus $\A$. We get an identity isotopy $\hat I=(\hat g_t)_{t\in[0,1]}$ on $\A$ by projection, as well a homeomorphism $\hat g=\hat g_1$.}
		
		And
		$\hat g(\hat w + (0,p))=\hat g(\hat w)+(0,p),$ for all $p\in\Z$, because the following diagram commutes (remember Notation \ref{not:notacao}) 
		\begin{equation*}
			\begin{tikzcd}
				\mathbb{R}^2 \arrow[rr, "\check g"] \arrow[rd, "\check\tau"] \arrow[dd, "\check\pi"'] &                                                                & \mathbb{R}^2 \arrow[rd, "\check\tau"] \arrow[d, no head] &                               \\
				& \mathbb{A} \arrow[rr, "\hat g\hspace{1.5cm}"] \arrow[ld, "\hat\pi"] & {} \arrow[d]                                          & \mathbb{A} \arrow[ld, "\hat\pi"] \\
				\mathbb{T}^2 \arrow[rr, "g"]                                                 &                                                                & \mathbb{T}^2                                          &                              
			\end{tikzcd} 
		\end{equation*}
		
		Moreover, it follows from the inequation (\ref{eq:deviations}) that 
		\begin{equation*}
			\left|\pr_2(\hat g^n(\hat w)-\hat w)\right|\leq M,\;\;\forall n\in\Z\;\;\mbox{and}\;\;\forall\check\tau(\check w)=\hat w\in\A
		\end{equation*}
		
		This means that for all points $\hat w\in\A$, the orbit of $\hat w$ is contained in a compact subset of $\A$. And, consequently, we have that the $\omega_g(\hat w)$ is non-empty, for all $\hat w\in\A$. Furthermore, by Corollary \ref{cor:nonwandering}, we have that $\hat g\in\homeo(\A)$ is non-wandering.

		By assumption $g\in\homeo_0(\T^2)$ has no topological horseshoe. Thus, by Remark \ref{remark:horseshoe} we have that $\hat g$ can not have a topological horseshoe. 
		
		So, $\hat g:\A\to\A$ is a non-wandering homeomorphism isotopic to identity such that the $\omega$-limit of any point $\hat w\in\A$ is non-empty and has no topological horseshoe. Then Theorem \ref{teo:A} implies that the map $\rot_{\check g}:\A\to\R$ that associates to each point $\hat w\in\A$ its rotation number $\rot(\hat w', \check g)$ is well defined and continuous. 	
		Let \[\hat G= \bigcup_{p\in\Z}\left( \bigcup _{r=0}^{q-1}\{\hat z_r+(0,p)\}\right),\] be the set of all lifts of the set $\{z, f(z), \cdots, f^{q-1}(z)\}$, where $\hat z_r\in\hat\pi^{-1}(f^r(z))$ for $r\in\{0,1,\cdots, q-1\}$. 
		
		Given any $\hat w\in\A$, Lemma \ref{lemma:omega-limit} implies that there exists some point $\hat w'\in\hat G$ such that $\hat w \in \omega_{\hat g}(\hat w')$. And by item (2) of Theorem \ref{teo:A} the rotation number of $\hat w$ must be the same as the rotation number of $\hat w'$.  But the set $\hat G$ is an enumerable set. Therefore, as the map $\rot_{\check g}$ is continuous, it must be constant.
		
		Which is a contradiction, because if we take $\hat z_0\in\hat\pi^{-1}(z_0)$ and $\hat z_a\in\hat\pi^{-1}(z_a)$ then $\hat z_0,\hat z_a\in\Ne^+(\hat g)$, as observed. Moreover, $\rot(\hat z_0,\check g)=\pr_1(\rho(z_0, \check g))=0$ and $\rot(\hat z_a,\check g)=\pr_1(\rho(z_a, \check g))=\pr_1(a)\neq 0$ because $a\in\rho(\check g)$ is a non-zero vector of a horizontal non-degenerate line segment of $\R$ containing the origin.
		
		Therefore, we have proved by contradiction that $f$ must have a topological horseshoe.
	\end{proof}

	\begin{theoremA}
		Let $f$ be a transitive homeomorphism of $\T^2$ and let $g$ be a power of $f$ such that $g$ is isotopic to the identity. If $g$ has both a fixed point and a non-fixed periodic point, then $f$ has a topological horseshoe.
	\end{theoremA}

	\noindent\textit{Proof.}\hspace{.1cm} Let $g=f^q\in\homeo_0(\T^2)$.	Let $z_0\in\T^2$ be a fixed point for $g$, that is $z_0\in\fix(g)$; $q_1>1$ be the minimal period of a non-fixed periodic point $z_1\in\T^2$ for $g$, that is $z_1\in\fix(g^{q_1})\setminus\fix(g)$. 
	
	As $f$ is transitive, let $z\in\T^2$ be a point such that $\omega_f(z)=\alpha_f(z)=\T^2$. Then by Lemma \ref{lemma:transitive} we have \[\bigcup_{r=0}^{q-1} \omega_g(f^r(z))=\bigcup_{r=0}^{q-1} \alpha_g(f^{-r}(z))=\T^2.\]
	
	So, there are some $r,s\in\{0,1,\cdots,q-1\}$ such that $z\in\omega_g(f^r(z))$ and $z\in\alpha_g(f^{-s}(z))$. But this imply that $z\in\omega_g(z)\cap\alpha_g(z)$. We will prove this fact for the $\omega$-limit, for $\alpha$-limit the proof is analogous.
	
	Note that if $z\in\omega_g(f^r(z))$ then \[f^r(z)\in\omega_g(f^{2r}(z))\Rightarrow\;\cdots\Rightarrow f^{(q-1)r}(z)\in\omega_g(f^{qr}(z)).\] This last $\omega_g$-limit set is equal to $\omega_g(z)$, because $g=f^q$ and it is $g$-invariant. So, we have \[z\in\omega_g(f^r(z))\subset\omega_g(f^{2r}(z))\subset\cdots\subset\omega_g(z),\] because is a fact that if $x\in\omega_g(y)$ then $\omega_g(x)\subset\omega_g(y)$.
	
	Let $\I_g$ be the set of identity isotopies for $g$, pick $I_0\in\I_g$ such that $z_0\in\fix(I_0)$ and let  $I\in\I_g$ be a maximal isotopy for $g$ such that $I_0\preceq I$.
	
	
	Let $\check{I}=(\check{g}_t)_{t\in[0,1]}$ be an identity isotopy that lifts $I$ to $\R^2$, $\check{g}:=\check{g}_1$ be the lift of $g$ to $\R^2$ given by the isotopy $I$.
	
	As $z_1\in\fix(g^{q_1})\setminus\fix(g)$ with $q_1> 1$ we know that the trajectory of $z_1$ by the isotopy $I$ until order $q_1$, $I^{q_1}(z_1)$, is a loop on $\T^2$. Then we have two possibilities for it: 
	\begin{enumerate}
		\item $I^{q_1}(z_1)$ is homotopic to zero on $\T^2$;
		\item Or not, which means that $I^{q_1}(z_1)$ is an essential loop on $\T^2$.
	\end{enumerate}
	
	We will proceed with the proof of the theorem considering these two cases separately. First we will show for case (2) and then for case (1).
	
	\subsection{Essential loop on $\T^2$}\label{subsection:essential-loop}
	
	{Suppose that $I^{q_1}(z_1)$ is an essential loop on $\T^2$ then} there exists $(p_1,p_2)\in\Z^2\backslash\{(0,0)\}$ such that \[\check{g}^{q_1}(\check{z}_1)=\check{z}_1+(p_1,p_2),\;\;\forall\check{z}_1\in\check\pi^{-1}(z_1).\]
	Moreover, $z_0\in\fix(I)$ implies $\check{g}(\check{z}_0)=\check{z}_0,$ for any $\check{z}_0\in\check\pi^{-1}(z_0).$
	
	This implies that $\rho(z_0, \check{g})=(0,0)$ and $\rho(z_1, \check{g})=\frac{1}{q_1}(p_1,p_2)$. And as $(p_1,p_2)\neq(0,0)$ we must have that $\rho(\check{g})$ has non-empty interior or it is a non-degenerate line segment of rational slope. 
	
	\begin{enumerate}
		\item { If $\operatorname{int}(\rho({\check{g}}))\neq\emptyset$ then Theorem 1 due to Llibre and Mackay (\cite{llibre1991rotation}) together with the fact that all rational vectors in $\operatorname{int}(\rho({\check{g}}))$ are realized, due to Franks \cite{franks1989realizing},  imply that $g$ has a topological horseshoe.}
		\item If $\rho(\check{g})$ is a segment of rational slope containing $(0,0)$ and $\frac{1}{q_1}(p_1,p_2)$, and moreover these vectors are realized by $z_0,z_1\in\T^2$, respectively, then Proposition \ref{prop:rational-rotationset} implies that $f$ has a topological horseshoe. 
	\end{enumerate}

	\subsection{Loop homotopic to zero on $\T^2$}
	
	Before proving this case, we present an easy result on closed transversal sections to topological flows on the plane that helps us to show that the set of leaves of an $\F$-transverse loop homotopic to zero in $\T^2$ is an inessential set that we prove next. 
	
	\begin{definition}[see \cite{strictly}]\label{def:essentialsets}
		An open subset $U$ of $\T^2$ is \textit{inessential} if every loop in $U$ is homotopic to zero in $\T^2$. Otherwise, $U$ is \textit{essential}. An arbitrary set $E$ of $\T^2$ is called inessential if it has some inessential neighborhood. And, finally, $E\subset\T^2$ is \textit{fully essential} if $\T^2\backslash E$ is inessential. 
	\end{definition}
	
	It follows immediately from the definition that if $U\subset\T^2$ is open and connected then $U$ is inessential if and only if, for every { $\check U\subset\check\pi^{-1}(U)$}, $\check U\cap\left(\check U+(p_1,p_2)\right)=\emptyset$ for any $(p_1,p_2)\in\Z^2\backslash\{(0,0)\}$. 
	
	{ Let $\check\F$ be a singular oriented foliation on $\R^2$,  ${\check\Gamma}:\T^1\to\dom(\check\F)$ be an $\check\F$-transverse simple loop on $\dom(\check\F)$ and $\check\gamma:\R\to\dom(\check\F)$ its natural lift. We will denote by $\inte({\check\Gamma})$ and $\exte({\check\Gamma})$ the bounded and unbounded connected components of $\R^2\backslash{\check\Gamma}$, respectively.
		
		Let $\check{U}=\bigcup_{t\in\R}\phi_{{\check\gamma}(t)}$ be the open topological annulus of the leaves that are crossed by ${\check\Gamma}$ on $\R^2$ (see Proposition \ref{prop:proposition2}).  Similarly, we will denote $\inte(\check{U})$ to the bounded connected component of $\R^2\backslash\check{U}$ and $\exte{(\check U)}$ to  the union of the unbounded connected components of $\R^2\backslash\check{U}$. It is worth noting that $\inte(\check{U})\subset\inte(\check{\Gamma})$ and it is a compact subset of $\R^2$.}
	
	\begin{lemma}\label{lema:conjlimite}
		Let $\check{\Gamma}$ be an $\check\F$-transverse simple loop in $\R^2$, $\check{\gamma}$ its natural lift and $\check{U}=\bigcup_{t\in\R}\phi_{\check{\gamma}(t)}$ the open topological annulus of the leaves that crosses $\check{\Gamma}$ in $\R^2$. Then one and only one of the following properties {holds}:
		\begin{enumerate}
			\item For all $\check \phi\subset \check U$, one has that $\omega(\check{\phi})\neq\emptyset$, $\omega(\check{\phi})\subset\inte(\check{U})$ and $\alpha(\check{\phi})\subset\exte(\check{U})$;
			\item For all $\check \phi\subset \check U$, one has that $\alpha(\check{\phi})\neq\emptyset$,  $\alpha(\check{\phi})\subset\inte(\check{U})$ and $\omega(\check{\phi})\subset\exte(\check{U})$.	
		\end{enumerate}
	\end{lemma}
	{\begin{proof}
			Fix, arbitrarily, a leaf $\check{\phi}$ contained in $\check U$. So $\check{\phi}\cap\check{\Gamma}=\{z\}$ and by Remark \ref{prop:1} 
			it is true that either $\check{\phi}^+_z\subset\inte(\check{\Gamma})$ and $\check{\phi}^-_z\subset\exte(\check{\Gamma})$ happen or the opposite happens, exclusively.

			If $\check{\phi}^+_z\subset\inte(\check{\Gamma})$ and $\check{\phi}^-_z\subset\exte(\check{\Gamma})$ then $\omega(\check{\phi})\subset\inte(\check{\Gamma})\cup\check{\Gamma}$ and $\alpha(\check{\phi})\subset\exte(\check{\Gamma})\cup\check{\Gamma}$.  
			By Poincar{\'e}-Bendixson Theorem, as $\inte(\check{\Gamma})\cup\check{\Gamma}$ is a compact set, neither $\omega(\check{\phi})\not\subset\check{U}$ nor $\alpha(\check{\phi})\not\subset\check{U}$, because there is no singularities nor closed leaves contained in $\check{U}$. Therefore \[\omega(\check{\phi})\subset\inte(\check{U})\quad\mbox{and}\quad \alpha(\check{\phi})\subset\exte(\check U).\]
			
			Let  $\check{\phi}'\subset\check U\setminus\{\check\phi\}$ be any other leaf such that $\check{\phi}'\cap\check{\Gamma}=\{z'\}$. So, there exist $t,t'\in\R$ such that $t<t'<t+1$ and $\check\gamma(t)=z$ and $\check\gamma(t')=z'$, where $\check\gamma$ is the natural lift of $\check\Gamma$. 
			By transversality of the path $\check{\Gamma}$ we note that for every $s\in[t,t']$, if $\phi^{+}_{\check\gamma(s)}$ is contained in $\inte(\check\Gamma)$, then the same holds for every $s'$ sufficiently close to $s$. Since $[t,t']$ is a compact interval, we get that $\check{\phi}'^+_{z'}\subset\inte(\check{\Gamma})$. Which implies, in an analogous way, that case (1) holds for leaf $\phi'$. 
			
			Therefore we conclude that if $\check{\phi}^+_z\subset\inte(\check{\Gamma})$ for some leaf $\check\phi\subset\check U$ then case (1) holds for all other leaves contained in $\check U$. 
			
			The case (2) is shown in an analogous way, assuming that $\check{\phi}^-_z\subset\inte(\check{\Gamma})$ and $\check{\phi}^+_z\subset\exte(\check{\Gamma})$ happens.	 
	\end{proof}}
	
	\begin{proposition}\label{prop:minha1}
		Let $\Gamma:\T^1\to\T^2$ be an $\F$-transverse loop homotopic to zero on $\T^2$ and $\gamma:\R\to\T^2$ its natural lift. Suppose $\Gamma$ has no $\F$-transverse self-intersection. Let $\check{\Gamma}:\T^1\to\R^2$ be a lift of the loop $\Gamma$ to $\R^2$  and  $\check{\gamma}:\R\to\R^2$ be the natural lift of $\check\Gamma$. We have the following 
		\begin{enumerate}
			\item  $\check\Gamma$ is an $\check\F$-transverse loop on $\R^2$;
			\item There exists an $\check\F$-transverse simple loop $\check{\Gamma}'$ such that  its natural lift $\check{\gamma}':\R\to\R^2$ is equivalent to $\check\gamma$;
			\item $\check{U}=\bigcup_{t\in\R}\phi_{\check{\gamma}(t)}$ is a topological open annulus in $\R^2$;
			\item $\check{U}\cap\check{U}+(p_1,p_2)=\emptyset$ for every $(p_1,p_2)\in\Z^2\backslash\{(0,0)\}$.
		\end{enumerate}
	\end{proposition}
	
	\begin{proof}
		Let $\Gamma:\T^1\to\T^2$, $\gamma:\R\to\T^2$, and $\check{\Gamma}:\T^1\to\R^2$ be the paths given in the statement. As $\Gamma$ is an $\F$-transverse loop homotopic to zero, since the lift to $\R^2$ of loops homotopic to zero are also loops, and the lift of $\F$-transverse paths to $\R^2$ are $\check\F$-transverse paths then, lifting such homotopy, we get that $\check{\Gamma}$ is an $\check\F$-transverse loop {on $\R^2$, proving (1)}.
		
		By assumption, we have that  $\Gamma$ has no $\F$-transverse self-intersection which implies its lift $\check{\Gamma}$ also has no $\check{\F}$-transverse self-intersection. 	Therefore the properties $(2)$ and $(3)$ follow from Proposition \ref{prop:proposition2}.
		
		{	Suppose, for contradiction, that $(4)$ does not hold then let $(p_1,p_2)\in\Z^2\backslash(0,0)$ be some vector such that $\check{U}\cap\check{U}+(p_1,p_2)\neq\emptyset$. As $\check\Gamma\subset\check U$ is a loop that lifts $\Gamma$ to $\R^2$, the loop $\check\Gamma+(p_1,p_2)\subset\check U+(p_1,p_2)$ is also a lift of $\Gamma$ to $\R^2$.}
		
		By assumption, the loop $\Gamma$ has no $\F$-transverse self-intersection. So, in addition to any of the loops $\check\Gamma$ and $\check\Gamma+(p_1,p_2)$ not having an $\check\F$-transverse self-intersection, we also have that $\check{\Gamma}$ and $\check{\Gamma}+(p_1,p_2)$ have no $\check\F$-transverse intersection. Because otherwise there would be $s<t$ in $\R$ such that their natural lifts $\check{\gamma}$ and $\check{\gamma}+(p_1,p_2)$ have an $\check\F$-transverse intersection {at} $\check{\gamma}(s)=\check{\gamma}(t)+(p_1,p_2)$ which would imply that $$\gamma(s)=\check\pi(\check{\gamma}(s))=\check\pi(\check{\gamma}(t)+(p_1,p_2))=\check\pi(\check{\gamma}(t))=\gamma(t)$$ and $\gamma$ would have an $\F$-transverse self-intersection in $\gamma(s)=\gamma(t)$, which is a contradiction with the assumption.
		
		Thus, with the Remark \ref{prop:1} and Proposition \ref{prop:proposition1}, we have that there exist $\check\F$-transverse simple loops $\check{\Gamma}'$ and $\check{\Gamma}''$ that are disjoint and $\check\F$-equivalent to $\check{\Gamma}$ and $\check{\Gamma}+(p_1,p_2)$, respectively. Which means that $\check{\Gamma}'\subset\check{U}$ and $\check{\Gamma}''\subset\check{U}+(p_1,p_2)$ and moreover   \[\inte(\check{U})\subset\inte(\check{\Gamma}')\;\mbox{and}\;\inte(\check{U}+(p_1,p_2))\subset\inte(\check{\Gamma}'')\] where $\inte(\check{U}+(p_1,p_2))=\inte(\check{U})+(p_1,p_2)$.

		As $\check{U}\cap\check{U}+(p_1,p_2)\neq\emptyset$ and $\check{U}$ is a foliated set, there exists a leaf  $\check{\phi}_0\subset\check{\F}$ satisfying
		
		{\begin{equation}\label{eq:phi0}
				\check{\phi}_0\subset\check{U}\cap\check{U}+(p_1,p_2). 
		\end{equation}	}

		By  Lemma \ref{lema:conjlimite} we know that for all $\check \phi\subset\check U$ either $\omega(\check{\phi})\neq\emptyset$, $\omega(\check{\phi})\subset\inte(\check{U})$ and $\alpha(\check{\phi})\subset\exte(\check{U})$ or
		$\alpha(\check{\phi})\neq\emptyset$,  $\alpha(\check{\phi})\subset\inte(\check{U})$ and $\omega(\check{\phi})\subset\exte(\check{U})$.
		
		Let us suppose that \[\omega(\check{\phi}_0)\neq\emptyset,\; \omega(\check{\phi}_0)\subset\inte(\check{U})\;\;\mbox{and}\;\;\alpha(\check{\phi}_0)\subset\exte(\check{U}).\] The other case is analogous. 
		
		By property (\ref{eq:phi0}) follows that $\check{\phi}_0-(p_1,p_2)\subset\check{U}$,
		so {by Lemma \ref{lema:conjlimite}} we must have $\omega(\check{\phi}_0-(p_1,p_2))\subset \inte(\check{U})$ which implies \[\omega(\check{\phi}_0)\subset\inte(\check{U})+(p_1,p_2)=\inte(\check{U}+(p_1,p_2)).\]

		Therefore \[\omega(\check{\phi}_0)\subset\inte(\check{\Gamma}')\;\;\mbox{and}\;\;\omega(\check{\phi}_0)\subset\inte(\check{\Gamma}'').\] 
		
		From this and the assumption that $\check{\Gamma}'$ and $\check{\Gamma}''$ are disjoint follows that $$\check{\Gamma}'\subset\inte(\check{\Gamma}'')\;\;\mbox{or}\;\;\check{\Gamma}''\subset\inte(\check{\Gamma}').$$

		Let us suppose $\check{\Gamma}'\subset\inte(\check{\Gamma}'')$, so {\begin{equation}\label{eq:intUintG}
				\inte(\check{U})\subset\inte(\check{\Gamma}'').
		\end{equation} }
		
		But this is a contradiction. Let us explain why.		
		
		The set $\inte(\check{U})$ is a union of leaves and singularities. Thus, by property (\ref{eq:intUintG}), \begin{itemize}
			\item if $\check{\phi}'\subset\check\dom(\F)\cap\inte(\check{U})$ then $\check{\phi}'\subset\inte(\Gamma'')$ and, therefore $\check{\phi}'\cap\check{\Gamma}''=\emptyset$. So $\check{\phi}'$ is not contained in $\check{U}+(p_1,p_2)$, which imply 
			$\check\phi'\subset\inte(\check{U}+(p_1,p_2))$. 
			\item if  $w\in\sing(\check\F)\cap\inte(\check{U})$ then $w\in\inte(\check{\Gamma}'')$ and as $w\notin\check{U}+(p_1,p_2)$, because $\check{U}+(p_1,p_2)$ is a foliated set, we have  $w\in\inte(\check{U}+(p_1,p_2))$.
		\end{itemize}
		
		So $\inte(\check{U})\subset\inte(\check{U}+(p_1,p_2))=\inte(\check{U})+(p_1,p_2)$ and this is a contradiction, because $(p_1,p_2)$ is a non-zero vector on $\Z^2$ and $\inte(\check{U})$ is a non-empty compact subset of $\R^2$.
		
		{If} $\check{\Gamma}''\subset\inte(\check{\Gamma}')$ then, similarly, we can show that  $\inte(\check{U})+(p_1,p_2)\subset\inte(\check{U})$, which is a contradiction again.
		
		Therefore, we must have $\check{U}\cap\check{U}+(p_1,p_2)=\emptyset$, meaning that $\check U$ is an inessential set of $\T^2$.
	\end{proof}
	
	\noindent\textit{Continuation of the proof of Theorem A.}\hspace{.1cm} Let's remember, $f:\T^2\to\T^2$ is a transitive homeomorphism where $g=f^q$ is isotopic to identity for some $q\geq 1$, $z_0\in\fix(g)$, $z_1\in\fix(g^{q_1})\setminus\fix(g)$ with $q_1>1$, $z\in\T^2$ is the point such that $\omega_f(z)=\alpha_f(z)=\T^2$, $I$ is a maximal isotopy for $g$ such that $z_0\in\fix(I)$ and $\check{I}=(\check{g}_t)_{t\in[0,1]}$ is an identity isotopy that lifts $I$ to $\R^2$ such $\check{g}:=\check{g}_1$. And now we are in the case that  $I^{q_1}(z_1)$ is  homotopic to zero on $\T^2$.
	
	We will consider $\F$ as a transverse foliation to $I$ and $\check{\F}$ as the foliation of $\R^2$ that lifts $\F$.  Observe that $\check{\F}$ is transverse to $\check{I}$ and as $\fix(I)\neq\emptyset$ we have that  $\fix(\check{I})=\check\pi^{-1}(\fix(I))$ is non-empty, then $\check{\F}$ is a singular oriented foliation of $\R^2$. 
	
	Furthermore, we will suppose that any $\F$-transverse path	{has} no $\F$-transverse self-intersection, because if some  $\F$-transverse path does then Theorem \ref{teo:horseshoe} would imply that $g$ has a topological horseshoe and, consequently, $f$ would have a topological horseshoe, finishing the proof. Thus any lift to $\R^2$ of any $\F$-transverse path has no $\check\F$-transverse self-intersection also.
	
	{Let $\gamma_1:=I^{\Z}_{\F}(z_1)$ be a whole $\F$-transverse trajectory of the point $z_1$ such that lifts $\Gamma_1:=I^{q_1}_{\F}(z_1)$ a $\F$-transverse loop homotopic to $I^{q_1}(z_1)$.} 
	
	As $\Gamma_1$ is homotopic to $I^{q_1}(z_1)$ on $\dom(I)$, then it is an $\F$-transverse loop that is itself homotopic to zero on $\T^2$. Moreover it has no $\F$-transverse self-intersection, by assumption. 
	So, it follows from Proposition \ref{prop:minha1} that any lift $\check{\Gamma}_1$ of $\Gamma_1$ to $\R^2$ is $\check\F$-equivalent to a multiple of an $\check\F$-transverse simple loop and if $\check{\gamma}_1:\R\to\R^2$ is the natural lift of the loop $\check\Gamma_1$, $\check{U}_1=\bigcup_{t\in\R}\phi_{\check{\gamma}_1(t)}$ is a topological open annulus on $\R^2$ which is disjoint from all integers translates of it. 
	
	With those assumptions, we have the following claims:
	
	\begin{affirmation}\label{af:nonwandering}
		Let $z\in\T^2$ be a point such that $\omega_f(z)=\alpha_f(z)=\T^2$. Given any lift $\check z\in\check\pi^{-1}(z)$ then $\check z$ is non-wandering for $\check g$.
	\end{affirmation}
	\begin{proof}

		In this case, we know that $z\in\omega_g(z)\cap\alpha_g(z)$. So there are increasing sequences of integers $(m_k)_{k\in\N}$ and $(n_l)_{l\in\N}$ such that $m_k\nearrow+\infty$ and $n_l\nearrow+\infty$ and $g^{m_k}(z)\to z\;\;\mbox{and}\;\; g^{-n_l}(z)\to z.$ That is, we can assume, eventually by taking subsequences, that 
		\[\begin{array}{l}
			d_{\T^2}(g^{m_k}(z),z)<1/k,\;\;\forall k\in\N \quad\mbox{and}\vspace{.2cm}\\
			d_{\T^2}(g^{-n_l}(z),z)<1/l,\;\;\forall l\in\N
		\end{array}\]
		
		Let $\gamma:=I^{\Z}_{\F}(z)$ and $\phi\in\F$ such that $z\in\phi$. By Lemma \ref{lemma:17} we know that are real numbers {$t_k\nearrow+\infty$ and $t'_l\searrow-\infty$ such that for all $k,l\in\N$ 
			
			\begin{equation}\label{eq:ti_tj_phi}
				\begin{array}{l}
					d( \gamma(t_k), g^{m_k}(z))<1/k \,,\vspace{.2cm}\\
					d( \gamma(t'_l), g^{-n_l}(z))<1/l\;\;\mbox{and}\vspace{.2cm}\\
					\gamma(t_k),\,\gamma(t'_l)\in\phi.
				\end{array}
		\end{equation}}
		
		Now, fix arbitrarily $\check z\in\check\pi^{-1}(z)$. So, there are two sequences $(p_k)_{k\in\N},(v_l)_{l\in\N}\in\Z^2$  such that \[\begin{array}{l}
			d(\check g^{m_k}(\check z),\check z+p_k)<1/k,\;\;\forall k\in\N \quad\mbox{and}\vspace{.2cm}\\
			d(\check g^{-n_l}(\check z),\check z+v_l)<1/l,\;\;\forall l\in\N.
		\end{array}\]
		
		{We will assume there are subsequences \[(p_{k_i})_{i\in\N}\subset(p_k)_{k\in\N}\;\;\mbox{and}\;\; (v_{l_j})_{j\in\N}\subset(v_l)_{l\in\N}\] such that 
			\begin{equation}\label{eq:finito}
				||p_{k_i}||\to+\infty\;\;\mbox{and}\;\;||v_{l_j}||\to+\infty, 
			\end{equation}
			and argue to a contradiction.}
		
		So, suppose 
		\begin{equation}\label{eq:wandering}
			\begin{array}{l}
				d(\check g^{m_{k_i}}(\check z),\check z+p_{k_i})<1/k_i,\;\forall i\in\N \vspace{.2cm}\\
				d(\check g^{-n_{l_j}}(\check z),\check z+v_{l_j})<1/l_j,\;\forall j\in\N 
			\end{array}
		\end{equation}
		{and let $\check\gamma:=\check I^{\Z}_{\F}(\check z)$ and $\check \phi\subset\check\F$ such that $\check z\in\check\phi$.} Thus we can assume, eventually by taking subsequences, the property in (\ref{eq:ti_tj_phi}) 
		and  the inequalities in (\ref{eq:wandering}) 
		imply for the real numbers {$t_{k_i}\nearrow+\infty$ and $t'_{l_j}\searrow-\infty$} that  \[\check\gamma(t_{k_i})\in\check\phi+p_{k_i}\;\mbox{and}\;\check\gamma(t'_{l_j})\in\check\phi+v_{l_j}.\] 
		
		These consequences together with the properties in (\ref{eq:finito}) means the whole $\check \F$-transverse trajectory of $\check z$ meets, arbitrarily in the past and in the future, integers translates of $\check \phi$ more and more distant of $\check\phi$.

		Moreover, we can assume that $z_1\in\omega_g(z)$, because if it is not then by Lemma \ref{lemma:transitive} we would have some $r\in\{0,1,\cdots,q-1\}$ such that $z_1\in\omega_g(f^r(z))$. The last one imply that $f^{-r}(z_1)\in\omega_g(z)$ and instead of taking $z_1$ we would take its iterate $f^{-r}(z_1)$ that belongs to $\omega_g(z)$. 
		
		So, supposing that $z_1\in\omega_g(z)$, Lemma \ref{lemma:17} say there are $a<b$ and $t$ real numbers that \[\gamma|_{[a,b]}\;\;\mbox{is $\F$-equivalent to}\;\;\gamma_1|_{[t,t+1]}.\] 
		
		So, there is a lift $\check\gamma_1$ of $\gamma_1$ such that 
		\begin{equation}\label{eq:lift-draws}
			\check\gamma|_{[a,b]}\;\;\mbox{is $\check\F$-equivalent to}\;\;\check\gamma_1|_{[t,t+1]}, 
		\end{equation} 
		
		By the discussion, above this claim, we know that $\check\gamma_1$ is the natural lift of an $\check\F$-transverse simple loop $\check\Gamma_1$ (up to equivalence) and $\check U_1=\bigcup_{t\in\R}\phi_{\check\gamma_1(t)}$ is a topological open annulus in $\R^2$ which is disjoint from all integers translates of it.
		
		Thus, the union $\check U_1\cup\inte(\check U_1)$ contains $\check\phi$, the leaf that contains the point $\check z$ and no other integer translate of it. 
		
		So, there must be {$p_{1}\in(p_{k_i})$, $v_{2}\in(v_{l_j})$} integer vectors, such that {$||p_{1}||$ and $||v_{2}||$} are large enough so that {$\check\phi+p_{1}$ and $\check\phi+v_{2}$} are in $\exte(\check U_1)$. Moreover, there are {$\bar{t}\in(t_{k_i})$, $\bar{t'}\in(t'_{l_j})$ such that $\bar{t}<a<b<\bar{t'}$} and 
		\begin{equation}\label{eq:recorrente}
			\check\gamma(\bar{t})\in\check\phi+p_{1}\;\;\mbox{and}\;\;\check\gamma(\bar{t'})\in\check\phi+v_{2}. 
		\end{equation}
		
		Property (\ref{eq:lift-draws}) means that the path $\check{\gamma}$ draws the loop $\check\Gamma_1$,  which means that given $J^{\check{\gamma}}_{\check{\Gamma}_1}=\{t\in\R\;|\;\check{\gamma}(t)\in\check{U}_1\}$, there exists a drawing component $J\subset J^{\check{\gamma}}_{\check{\Gamma}_1}$ that contains $[a,b]$ and such that $\check{\gamma}|_J$ draws $\check{\Gamma}_1$.
		
		 Observe that {$\bar{t},\bar{t'}\not\in J^{\check{\gamma}}_{\check{\Gamma}_1}$} and they are separated by $[a,b]$, so $J$ must also separates them. Moreover, {$\check\gamma(\bar{t})$ and $\check\gamma(\bar{t'})$} are both in $\exte(\check U_1)$, by property (\ref{eq:recorrente}) and $\check\gamma$ has no $\check\F$-transverse intersection, by assumption, which is a contradiction with Corollary \ref{cor:prop20}.
		
		Thus, we conclude that property (\ref{eq:finito}) does not hold. So there must exist a constant real number $K$ or a constant real number $L$ such that \[||p_{k}||<K\;\;\mbox{or}\;\;||v_{l}||<L,\;\forall k,l\in\N.\] 
		
		Then, by the Pigeonhole Principle, either there is a subsequence $(p_{k_i})_{i\in\N}$ of $(p_k)_{k\in\N}$ such that $p_{k_i}=p\in\Z^2$ for all $i\in\N$ or there is a subsequence $(v_{l_j})_{j\in\N}$ of $(v_l)_{l\in\N}$ such that $v_{l_j}=v\in\Z^2$ for all $j\in\N$, so
		
		\[\begin{array}{l}
			d(\check g^{m_{k_i}}(\check z),\check z+p)<1/k,\;\;\forall i\in\N\;\;\mbox{or} \vspace{.2cm}\\
			d(\check g^{-n_{l_j}}(\check z),\check z+v)<1/l,\;\;\forall j\in\N
		\end{array}\] which implies that $\check z+p\in\omega_{\check g}(\check z)$ or $\check z+v\in\alpha_{\check g}(\check z)$ that is equivalent to $\check z\in\omega_{\check g}(\check z-p)$ or $\check z\in\alpha_{\check g}(\check z-v)$, because $\check g$ is a homeomorphism that commutes with the integer vector translations. Therefore, $\check z$ is  non-wandering for $\check g$.

		As $\check z\in\check\pi^{-1}(z)$ was taken arbitrarily, we conclude that any lift $\check z\in\check\pi^{-1}(z)$ is non-wandering.
	\end{proof}

	\begin{affirmation}\label{lema:UdomF}
		If $U_1\subset\T^2$ is the set of leaves crossed by the loop $\Gamma_1$ then 
		$U_1=\dom(\F)$.
	\end{affirmation}
	\begin{proof}
		
		Let $z\in\T^2$ be the point such that $\omega_f(z)=\alpha_f(z)=\T^2$. 
		
		By the above claim we know that any lift $\check z\in\check\pi^{-1}(z)$ of the point $z\in\T^2$ is non-wandering for $\check g$. Moreover, follows from its proof that there is a lift $\check\Gamma_1$ of the loop $\Gamma_1$ such that $\check\gamma:=\check I^{\Z}_{\F}(\check z)$ draws $\check\Gamma_1$. As, by assumption, $\check\gamma$  has no $\check\F$-transverse self-intersection, Corollary \ref{cor:prop23} and {Remark \ref{remark:scholium}} imply that $\check\gamma$ is contained in $\check U_1$. So, by projection, we have that the whole $\F$-transverse trajectory of $z$, $\gamma:=I^{\Z}_{\F}(z)$ is contained in $U_1$.
		
		But we know that a whole $\F$-transverse trajectory of a point whose $\omega,\alpha$-limits sets are the whole surface meets all leafs of $\dom(\F)$.
		
		So, we must have that $U_1=\dom(\F)$.
	\end{proof}
	
	This statement implies that the frontier $\partial U_1$ of $U_1$ is contained in the singularity set $\sing(\F)$. But more than that, we have $\partial U_1=\sing(\F)$ because the whole orbit of $z$ is contained in $U_1$ and also the orbit of $z$ is dense in $\T^2$.
	
	Moreover, Proposition \ref{prop:minha1} {says} that any lift $\check{U}_1\in\check\pi^{-1}(U_1)$ is disjoint from any of its integer translations and so $U_1=\dom(\F)$ is an inessential set and $\partial U_1=\sing(\F)$ is a {fully} essential set.
	
	Let us fix, arbitrarily, {$\check U_1\subset\check\pi^{-1}(U_1)$}, { so $\partial\check{U}_1\subset\sing(\check{\F})$, which implies that $\check U_1$ is a $\check g$-invariant set. Moreover, follows from  the proof of Lemma \ref{lema:conjlimite} that there exists at least one singularity of the foliation $\check{\F}$ that belongs to $\inte(\check{U}_1)$, we will denote it by $\check{z}'$.} 
	
	Let us consider the annulus $A=\R^2\backslash\{\check{z}'\}$.	As $\check z'\in\sing(\check \F)$,  the restriction   $\check g_A:=\check g|_{A}$ is well defined and it is a homeomorphism isotopic to identity, where $\check I_A:=\check I|_{A}$ and $\check\F_A:=\check \F|_{A}$ are a maximal identity isotopy for $\check g_A$ and a transverse foliation to $\check I_{A}$ of $A$, respectively, such that $\sing(\check\F_A)=\sing(\check\F)\setminus\left\{{\check z'}\right\}$. 
	
	Let {$\check\Gamma_1\subset\check\pi^{-1}(\Gamma_1)\cap\check U_1$}. As $\check{z}'\in\inte(\check{U}_1)\subset\inte(\check{\Gamma}_1)$, we have that $\check{\Gamma}_1$ is an essential loop in $A$. Thus $\check{z}_1\in\Ne^+(\check{g}_A)$ and  $\rot(\check{z}_1,\check{g}_A)$ is a non-zero rational number.
	
	Furthermore, as $\sing(\check{\F})$ is the lift of a totally essential set of $\T^2$ we have that there exists $\check{z}_2\in\sing(\check\F_A)\cap\partial\check{U}_1$, so  $\check{z}_2\in\Ne^+(\check{g}_A)$ and  $\rot(\check{z}_2,\check{g}_A)=0$.
	
	As we explain before, we can assume that $z_1\in\omega_g(z)$. We also can assume that $\check\pi(\check z_2)=z_2\in\omega_g(z)$. Because if it is not, again by Lemma \ref{lemma:transitive} there would be some $r\in\{0,1,\cdots,q-1\}$ such that $f^{-r}(z_2)\in\omega_g(z)$ and as $z_2\in\sing(\F)$ then also its iterate $f^{-r}(z_2)\in\sing(\F)$. So we would take this iterate instead.
	
	{Thus we must have that $\check z_1,\check z_2\in\omega_{\check{g}}(\check z)$. Indeed, if $z_i\not\in\omega_{\check g}(\check z)$, for $i=1,2$, there would be  $p^i_k\in\Z^2\setminus\{(0,0)\}\;(i=1,2)$ such that $d(\check g^{m_k}(\check z),\check z_i+p^i_k)<1/k$, for all $k\in\N$ (eventually by taking a subsequence). But this is impossible, since $\check U_1$ is $\check g_A$-invariant and does not intersect any of its integer vector translate}. So $\check z\in\Ne^+(\check g_A)$.

	{However we have that $\rot(\check{z}_1,\breve{g}_{\breve A})\neq\rot(\check{z}_2,\breve{g}_{\breve A})$, where $\breve g_A$ is the lift of $\check g_{\breve A}$ to the universal covering space $\breve A$ of $A$ given by the isotopy $I_A$.} So, the contrapositive of Theorem \ref{teo:A} implies that $\check{g}_A$ possesses a topological horseshoe $\check{\Delta}$ as defined in the introduction. And, as we have that $\check{g}=\check{g}_A$ on $\R^2\backslash\{\check{z}'\}$ and $\check g(\check z')=z'$ then $\check{\Delta}$ is also a topological horseshoe for $\check{g}$. And, therefore, by Remark \ref{remark:horseshoe}, we have that $\Delta=\check\pi(\check{\Delta})$ is a topological horseshoe for $g$, as defined at the introduction. Finally, as $g=f^q$ then $\Delta$ is also a topological horseshoe for $f$.
	
	\section{Proof of Theorem B}\label{sec:thmB}
	{Before we restate and prove Theorem B, let us introduce some fundamental concepts about homeomorphisms isotopic to Dehn Twist first.}
	
	It is a well known fact that if $f\in\homeo(\T^2)$ is isotopic to a Dehn twist and if $\check{f}\in\homeo(\R^2)$ is a lift of $f$ then there is some $A\in\operatorname{SL}(2,\Z)$ that is conjugate to matrix $\begin{pmatrix}
		1 & m \\ 0 & 1 
	\end{pmatrix}$ where $m\in\Z\backslash\{0,0\}$ and such that for all $(p_1,p_2)\in\Z^2$
	\begin{equation*}
		\check{f}(\check{z}+(p_1,p_2))=\check{f}(\check{z})+A\begin{pmatrix}
			p_1\\p_2
		\end{pmatrix}.
	\end{equation*}
	
	As the proofs that we will presented in this chapter are preserved by conjugation, we will consider only the case where $A$ is in this special {form}. So

	\begin{equation*}
		A=\begin{pmatrix}
			1 & m \\ 0 & 1
		\end{pmatrix},\end{equation*} where $m\in\Z\backslash\{0\}$. We will denote the induced map by $A$ in $\T^2$ with the letter $A$ subscripted: $f_A$.

	As we establish in Notation \ref{not:notacao}, let $\check{\pi}:\R^2\to\T^2$ be the canonical universal covering of $\T^2$ and $\check\tau:\R^2\to\A$ the canonical universal covering of $\A$ and {$\hat{\pi}:\A\to\T^2$} a covering map such that $\hat{\pi}\circ\check\tau=\check{\pi}$.
	
	We will use the definition of vertical rotation set for homeomorphisms of $\T^2$ isotopic to a Dehn twist given in \cite{addas2002} and also in \cite{doeff1997rotation}. 
	\begin{definition}
		Let $f:\T^2\to\T^2$ be a homeomorphism isotopic to a Dehn twist and fix a lift $\hat{f}:\A\to\A$ of $f$ to $\A$. \textit{The vertical rotation set of the lift $\hat{f}$}, namely $\rho_V(\hat{f})$, is defined as the set \[\rho_V(\hat{f})=\bigcap_{k\geq1}\overline{\bigcup_{n\geq k}\left\{\dfrac{\pr_2\left(\hat{f}^n(\hat{z})-\hat{z}\right)}{n}\;:\;\hat{z}\in\A\right\}}.\] 
	\end{definition}
	If there is $\hat z\in\A$ such that the limit \[\lim_{n\to\infty}\frac{\pr_2\left(\hat f^n(\hat z)-\hat z\right)}{n}\;\;\mbox{exists}\] then we will say that $\hat z\in\A$ has a \textit{vertical rotation number} and we will denote it by $\rho_V(\hat z,\hat f)$. Furthermore, observe that the restriction to the second coordinate implies that for all $p\in\Z$, \[\rho_V(\hat{z}+(0,p),\hat f)=\rho_V(\hat{z}, \hat f),\] whenever $\hat{z}\in\A$ has a vertical rotation number. This implies that it only depends of the point $z=\hat{\pi}(\hat{z})\in\T^2$. So, we can use de notation $\rho_V(z,\hat f)$, where $z\in\T^2$, instead.
	
	Regarding the dependence of the lift, the vertical rotation set preserves some good properties, as in the bi-dimensional case:
	{\begin{proposition}\label{prop:vertical-rotationset}
			Let $f\in\homeo(\T^2)$ be isotopic to a Dehn twist and fix a lift $\hat{f}\in\homeo(\A)$ of $f$. Then,  
			\begin{enumerate}
				\item $\rho_V(\hat{f})$ is a non-empty and compact interval of $\R$; 
				\item $\rho_V(\hat{f}^q+(0,p))=q\rho_V(\hat{f})+p$, where $q\in\Z$ and $p\in\Z$.
			\end{enumerate}
	\end{proposition}}
	
	For other properties, see \cite{addas2002}, \cite{addas2005}  and \cite{doeff1997rotation}.
	
	As before, we can define the vertical rotation number associated to an invariant measure.
	\begin{definition}\label{def:rho_mu}
		Let $f\in\homeo(\T^2)$ be isotopic to a Dehn twist and fix a lift $\hat{f}\in\homeo(\A)$ of $f$. For an $f$-invariant Borel probability measure $\mu$, the vertical rotation number associated to $\mu$ is defined as \[\rho_V(\mu)=\int_{\T^2}\varphi\operatorname{d}\mu,\] where the vertical displacement function $\varphi:\T^2\to\T^1\times\R$ is given by $\varphi(z)=\pr_2(\hat{f}(\hat{z})-\hat{z})$, where $\hat{z}\in\hat{\pi}^{-1}(z)$.  Note that this definition does not depend on the choice of $\hat z$, only of $z\in\T^2$. 
	\end{definition}
	
	By Birkhoff's Ergodic Theorem we have that if $\mu$ is an $f$-ergodic Borel probability measure such that $\rho_V(\mu)=a$ then for $\mu$-almost every point $z\in\T^2$ and any $\hat{z}\in\hat{\pi}^{-1}(z)$, \[\lim_{n\to\infty}\dfrac{\pr_2\left(\hat{f}^{n}(\hat{z})-\hat{z}\right)}{n}=a.\]
	
	As $\rho_V$, in Definition \ref{def:rho_mu}, is a continuous linear functional on $\mathcal{M}_{\T^2}(f)$, the set of the $f$-invariant Borel probability {measures}, and as $\mathcal{M}_{\T^2}(f)$ is compact and convex in the weak-$\star$ topology  we must have that \[\rho_V(\mathcal{M}_{\T^2}(f))=[a,b]\subset\R,\] where $a=\displaystyle\inf_{\mu\in\mathcal{M}_{\T^2}(f)}\;\rho_V(\mu)$ and $b=\displaystyle\sup_{\mu\in\mathcal{M}_{\T^2}(f)}\;\rho_V(\mu)$, and it is possible that $a=b$.
	
	Moreover, it is possible to show that $\rho_V(\hat{f})=\rho_V(\mathcal{M}_{\T^2}(f))=[a,b]$, { see \cite{doeff1997rotation}.} Thus, using the fact that $\mathcal{M}_{\T^2}(f)$ is compact and convex and its extremal points are ergodic measures we get the next proposition:
	\begin{proposition}\label{prop:measure-vertical}
		There exist two $f$-ergodic Borel probability measure $\mu_a$ and $\mu_b$ on $\T^2$ such that $\rho_V(\mu_a)=a$ and $\rho_V(\mu_b)=b$.
	\end{proposition}

	Finally, we state a result due to Addas-Zanata, Tal and Garcia which {ensures} that if $\rho_V(\hat f)=\{0\}$ then there exists $\hat K\subset\A$ an {essential $\hat f$-invariant continuum}.
	
	\begin{theorem}[Theorem $2$ in \cite{braulio}]\label{teo:braulio}
		Given $f\in\homeo(\T^2)$ isotopic to an $f_A$ and a lift $\hat f\in\homeo(\A)$, if $\rho_V(\hat f)=\left\{\frac{r}{s}\right\}$ for some $r\in\Z$ and $s\in\N$, then there exists a compact connected set $\hat K\subset \A$, invariant under $\hat f^s-(0,r)$, which separates the ends of the vertical annulus.
	\end{theorem}
	
	Now, we are able to prove Theorem B.
	
	{\begin{theoremB}
			Let $f$ be a transitive homeomorphism of $\T^2$ and let $g$ be a power of $f$ such that $g$ is isotopic to a Dehn twist. If $g$ has a periodic point, then $f$ has a topological horseshoe.
	\end{theoremB}}
	\noindent\textit{Proof.}\hspace{.1cm} Let $g=f^{q}$, where $q\geq 1$, be the positive power of $f$ that is isotopic to a Dehn twist. So, there is a matrix $A\in\operatorname{SL}(2,\Z)$ \[A=\begin{pmatrix}
		1 & m \\ 0 & 1
	\end{pmatrix},\] where $m\in\Z\backslash\{(0,0)\}$ and such that $g$ is conjugate to $g_A$, a homeomorphism isotopic to the Dehn twist map induced by the matrix $A$. But, as we establish at the begging of this section  we will assume that $g$ is itself isotopic to $g_A$. This means that if $\check g\in\homeo(\R^2)$ is a lift of $g$ to $\R^2$ then \[\check g(\check z+(p_1,p_2))=\check g(\check z)+A\begin{pmatrix}
		p_1\\p_2
	\end{pmatrix}.\]

	{Let $z_0\in\T^2$ be a periodic point for $g$ and $q_0\geq 1$ be the period of $z_0$ for $g$. Take $h=g^{q_0}$. So, $h$ is itself a homeomorphism isotopic to a Dehn twist and {it fixes} $z_0\in\T^2$.
		
		Let $\check h\in\homeo(\R^2)$ be a lift of $h$ such that $\check h(\check z_0)=\check z_0$, {for some $\check z_0\in\check\pi^{-1}(z_0)$}. 
		
		Moreover, we have that $\check h$ induces a lift $\hat h\in\homeo(\A)$ of $h$ to {$\A=\T^1\times\R$} such that the following diagram commutes:}
	
	\begin{equation}\label{diagram:dehn}
		\begin{tikzcd}
			\mathbb{R}^2 \arrow[rr, "\check h"] \arrow[rd, "\check\tau"] \arrow[dd, "\check\pi"'] &                                                                & \mathbb{R}^2 \arrow[rd, "\check\tau"] \arrow[d, no head] &                               \\
			& \mathbb{A} \arrow[rr, "\hat h\hspace{1.5cm}"] \arrow[ld, "\hat\pi"] & {} \arrow[d]                                          & \mathbb{A} \arrow[ld, "\hat\pi"] \\
			\mathbb{T}^2 \arrow[rr, "h"]                                                 &                                                                & \mathbb{T}^2                                          &                              
		\end{tikzcd}
	\end{equation}
	
	Take this homeomorphism $\hat h\in\homeo(\A)$ induced by $\check h$, and let $\rho_V(\hat h)$ be its vertical rotation interval. In addition, {as $\check h(\check z_0)=\check z_0$, for some $\check z_0\in\check\pi^{-1}(z_0)$, we have that $\hat h(\hat z_0)=\hat z_0$, for any $\hat z_0\in\hat\pi^{-1}(z_0)$.} Therefore $0\in \rho_V(\hat h)$.  
	
	So, we have two possibilities:
	\begin{enumerate}
		\item $\rho_V(\hat h)=\{0\}$ or,
		\item $\rho_V(\hat h)$ is a non-degenerate interval of $\R$ containing $\{0\}$. 
	\end{enumerate}

	\subsection{If $\rho_V(\hat h)=\{0\}$}
	In this case, by Theorem \ref{teo:braulio}, there exists $\hat K\subset\A$ an {essential $\hat h$-invariant continuum}. So, there exists an integer $M_0>0$ such that $\hat K\subset\T^1\times[-M_0,M_0]$. Let $\hat K^-$ and $\hat K^+$ be the connected components of $(\A)\backslash \hat K$ that contains $-\infty$ and $+\infty$, respectively. As $\hat K$ is $\hat h$-invariant, then $\hat K^-$ and $\hat K^+$ are also $\hat h$-invariant. Then, there is $M>M_0$ such that $\hat K^-\cap(\hat K^+-(0,M))$ is an open, connected and $\hat h$-invariant set that contains a fundamental domain of $\T^2$. See Figure \ref{fig:K}.
	
	\begin{figure}[!htb]
		\centering
		\def\svgwidth{3cm}
		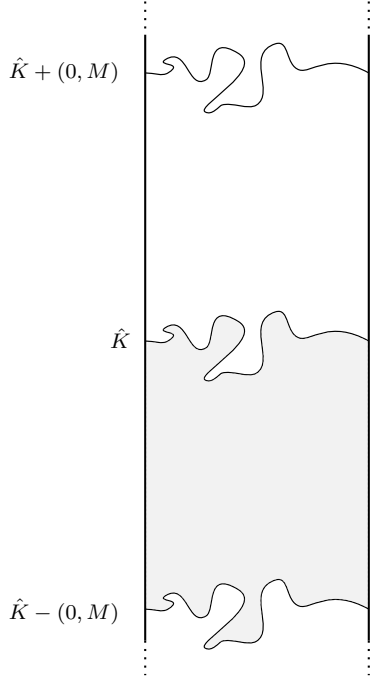
		\caption{Existence of a continuum $\hat K$.}
		\label{fig:K}
	\end{figure}
	
	Therefore, for all $\hat w\in\A$ and all integer $n>0$ there is some $M'>M$ such that
	\begin{equation}\label{eq:20}
		|\pr_2(\hat h^n(\hat w)-\hat w)|<M'.
	\end{equation} 
	
	As observed in the proof of Proposition \ref{prop:rational-rotationset}, this means that for all points $\hat w\in\A$, the $\omega_{\hat h}(\hat w)$ is not empty.
	
	As $h\in\homeo(\T^2)$ is isotopic to Dehn Twist and the diagram in (\ref{diagram:dehn}) commutes, we have that {$\hat h(\hat w +(p,0))=\hat h(\hat w)+(p,0)$, for all $p\in\Z$.}

	Since  $\check h\in\homeo(\R^2)$ commutes with the integer horizontal translations then homotopy theory implies that $\hat h\in\homeo(\A)$ is homotopic to identity. Moreover $\hat h$ preserves the ends of $\A$. Thus, by \cite{bclr16} (see propositions $3.3$, $3.6$ and $8.1$), $\hat h$ is isotopic to identity.

	Moreover, as $f$ is topologically transitive, { it is possible to apply Corollary \ref{cor:nonwandering}  because of (\ref{eq:20})}. So, we have that $\hat h$ is non-wandering.

	As we did in the proof of Proposition \ref{prop:rational-rotationset}, if we suppose, by contradiction, that $f$ has no topological horseshoe, then $h$ and also $\hat h$ have no topological horseshoe. And moreover, the map $\rot_{\check h}:\A\to\R$ is well-defined, continuous and must to be constant.
	
	However this is impossible. Let us explain why. We know that $h\in\homeo(\T^2)$  is isotopic to a Dehn twist map induced by \[A=\begin{pmatrix}
		1 & m \\ 0 & 1
	\end{pmatrix},\] and the diagram in (\ref{diagram:dehn}) commutes. So we must have that $\check h\in\homeo(\R^2)$ is also a lift of $\hat h\in\homeo(\A)$ to $\R^2$ and for all $\check z\in\R^2$
	\begin{equation}\label{eq:check g}
		\check h(\check z+(p_1,p_2))=\check h(\check z)+(p_1+mp_2,p_2),\quad\mbox{where}\;(p_1,p_2)\in\Z^2.
	\end{equation} 
	
	Moreover, we know  there is $\hat z_0\in\A$ such that \[\hat h(\hat z_0)=\hat z_0,\] then we can assume that $\rot(\hat z_0,\check h)=0$, because if it is not then we can take some horizontal integer translation of $\check h$ that satisfies this property. And, because of (\ref{eq:check g}), if we take $\hat z_1=\hat z_0+(0,1)$ then \[\check h(\check z_1)=\check z_1+(m,0),\;\check z_1\in\check\tau^{-1}(\hat{z}_1)\] and then $\rot(\hat z_1,\check h)=m\neq0$. And this is a contradiction with the fact that $\rot_{\check h}$ is a constant map.
	
	So, for this case, we prove, by contradiction, that $f$ must to have a topological horseshoe. 
	
	\subsection{If $\rho_V(\hat h)$ is a non-degenerate compact interval of $\R$ containing $\{0\}$}
	
	The existence of a horseshoe for this case is already known, althought we could not find it explicitly written in the literature. In any case it follows as a scholium of Theorem 3.1 due to Doeff and Misiurewicz in \cite{doeff1997shear} together with the techniques developed by Llibre and Mackay in \cite{llibre1991rotation}. Furthermore, the case where $f$ is a $C^{1+ \epsilon}$ diffeomorphism has already been proved by Addas-Zanata in \cite{salvador2015}.
	
	In the following we will give a new proof for the existence of a topological horseshoe for this case only using Forcing Theory and Rotation Theory.
	
	\begin{definition}
		Let $\hat f\in\homeo(\A)$ be isotopic to the identity and let $\check f\in\homeo(\R^2)$ be one of its lifts. We say that:
		\begin{enumerate}
			\item $\hat f$ satisfies the \textit{{infinity} twist condition (ITC)} if for any lift $\check f\in\homeo(\R^2)$ the following property holds: \[\pr_1\circ\check f(\check x,\check y)\to\pm\infty\;\;\;\mbox{as}\;\;\check y\to\pm\infty.\] 
			\item $(\hat f,\check f)$ has a periodic orbit of rotation number $\frac{r}{s}$, if there exists $\hat z\in\A$ such that $\hat f^s(\hat z)=\hat z$ and $\check f^s(\check z)=\check z +(r,0)$ for all $\check z\in\check\tau^{-1}(\hat z)$.
		\end{enumerate}
	\end{definition}
	
	Note that if $\hat f\in\homeo(\A)$ is a lift of $f\in\homeo(\T^2)$ homeomorphism isotopic to a Dehn twist map induced by the special matrix $A$ then $\hat f$ is isotopic to identity (as explained before) and satisfies ITC property. Indeed, if $\check f\in\homeo(\R^2)$ is a lift of $\hat f$, then $\check f$ is also a lift of $f$ satisfying \[\check f(\check z +(p_1,p_2))=\check f(\check z)+(p_1+mp_2,p_2),\] where $m\geq 1$ is given by the Dehn Twist map.
	
	\begin{theorem}[Theorem $1$ at \cite{addas2005}]\label{teo:salvador}
		Let $\hat f\in\homeo(\A)$ be isotopic to the identity and let $\check f\in\homeo(\R^2)$ be some lift of $f$. Suppose $\hat f$ satisfies the ITC property and there exist points $\hat z_1,\hat z_2\in\A$ such that \[\begin{array}{cl}
			\pr_2\circ\hat f^n(\hat z_1)\to\mp\infty &\;\;\;\mbox{as}\;\;n\to\pm\infty\;\;\mbox{and}\\
			\pr_2\circ\hat f^n(\hat z_2)\to\pm\infty &\;\;\;\mbox{as}\;\;n\to\pm\infty,\\
		\end{array}\]
		then for all rational numbers $\frac{r}{s}$, $(\hat f,\check f)$ has at least two periodic orbits with rotation number equal to $\frac{r}{s}$.
	\end{theorem}

	\noindent\textit{Continuation of the proof of Theorem B.}\hspace{.1cm} Let $h$, $\check h$ and $\hat h$ be as in the beginning of this proof. That is $h\in\homeo(\T^2)$ is isotopic to a Dehn twist map induced by \[A=\begin{pmatrix}
		1 & m \\ 0 & 1
	\end{pmatrix},\] {$\check h\in\homeo(\R^2)$ is a lift of $h$ that fix some lift $\check z_0\in\check\pi^{-1}(z_0)$ of the fixed point $z_0\in\T^2$ for $h$} and $\hat h\in\homeo(\A)$ is a lift of $h$ induced by $\check h$ such that the diagram in (\ref{diagram:dehn}) commutes. 
	
	Then, as explained before, we know that $\hat h\in\homeo(\A)$ is isotopic to identity and, moreover, $0\in\rho_V(\hat h)$. {Now, we are assume that $\rho_V(\hat h)$ is a non-degenerate {compact} interval of $\R$ containing $\{0\}$. 
		
	We can assume, without loss of generality, that $\rho_V(\hat h)=[a,b]$ where $a<0<b$. Because it is always possible choose a lift with such property.
		
		Indeed, if $0\notin\operatorname{int}(\rho_V(\hat h))$ then $0\in\partial(\rho_V(\hat h))$. Suppose that $\rho_V(\hat h)=[0,c]$ (the other case, namely $\rho_V(\hat h)=[c,0]$, is analogous). Take any $\frac{r}{s}\in(0,c)$, where $r\in\N$ and $s\in\N$. So, $h'= h^s$ is a homeomorphism in $\T^2$ isotopic to a Dehn twist map induced by $A$ and such the lift $\hat h':=\hat h^s-(0,r)$ to $\A$ has the property that $0\in\operatorname{int}(\rho_V(\hat h'))$, once $\rho_V(\hat h')=s\rho_V(\hat h)-r$. 
	}
	
	For this reason, let $\hat h$ be such that $\rho_V(\hat h)=[a,b]$, where $a<0<b$. Proposition \ref{prop:measure-vertical} and Birkhoff's Ergodic Theorem imply that there exist $z_a,z_b\in\T^2$ such that 
	\begin{equation}\label{eq:z_az_b}
		\begin{array}{ccc}
			\rho_V(z_a,\hat h)=a & \mbox{and} & \rho_V(z_a, \hat h^{-1})=-a\\
			\rho_V(z_b, \hat h)=b & \mbox{and} & \rho_V(z_b, \hat h^{-1})=-b
		\end{array}
	\end{equation}
	
	{The following lemma is an easy consequence for points satisfying the equations above, and the proof is omitted.}
	
	\begin{lemma}\label{lema:teo_salvador}
		Let $z_a,z_b\in\T^2$ be the points that satisfy the equations in (\ref{eq:z_az_b}). If $\hat z_a\in\hat\pi^{-1}(z_a)$ and $\hat z_b\in\hat\pi^{-1}(z_b)$ then 
		\[\begin{array}{cl}
			\pr_2\circ\hat h^n(\hat z_a)\to\mp\infty &\;\;\;\mbox{as}\;\;n\to\pm\infty\;\;\mbox{and}\\
			\pr_2\circ\hat h^n(\hat z_b)\to\pm\infty &\;\;\;\mbox{as}\;\;n\to\pm\infty.\\
		\end{array}\]
	\end{lemma}
	%
	%
	
	Let $\check h\in\homeo(\R^2)$ be a lift of $\hat h$. As $\hat h$ is a lift of a homeomorphism isotopic to a Dehn twist map induced by the special matrix $A$ in $\T^2$, we have that $\hat h$ satisfies \textit{ITC}. And together with the Lemma \ref{lema:teo_salvador} we have, by Theorem \ref{teo:salvador}, that for all rationals $\frac{r}{s}\in[a,b]$, $(\hat h, \check h)$ has a periodic orbit with rotation number $\frac{r}{s}$.
	
	So, there are $\hat w,\hat w'\in\A$, $r,r'\in\Z$ and $s,s'\in\N$ such that $\frac{r}{s}\neq\frac{r'}{s'}$ and \[\begin{array}{ll}
		\hat h^s(\hat w)=\hat w, & \check h^s(\check w)=\check w +(r,0), \;\mbox{where}\; \check w\in\check\tau^{-1}(w)\\
		\hat h^{s'}(\hat w')=\hat w', & \check h^{s'}(\check w')=\check w' +(r',0), \;\mbox{where}\; \check w'\in\check\tau^{-1}(w').
	\end{array}\]	
	This means that $w,w'\in\Ne^+(\hat h)$, $\rot(\hat w,\check h)=\frac{r}{s}$ and $\rot(\hat w',\check h)=\frac{r'}{s'}$.
	
	Now, let $\hat h_{\sphere}$ be the continuous extension of $\hat h$ to $\mathbb{S}^2=\A\cup\{N,S\}$, then Lemma \ref{lema:teo_salvador} implies that \[\begin{array}{c}
		(\omega)_{\hat h_{\sphere}}(\hat z_a)=(\alpha)_{\hat h_{\sphere}}(\hat z_b)=\{S\}\\
		(\alpha)_{\hat h_{\sphere}}(\hat z_a)=(\omega)_{\hat h_{\sphere}}(\hat z_b)=\{N\}
	\end{array}\]
	So, follows that $N$ and $S$ are in the same Birkhoff recurrence classes of $\hat h_{\sphere}$. 
	
	Therefore, Proposition \ref{prop:D} implies that $\hat h$ has a topological horseshoe. And, as explained before, it follows that $h=g^{q_0}$,  $g=f^q$ and therefore $f$, have a topological horseshoe, as we wanted to prove.\qed
	
	\singlespacing   
	\bibliographystyle{amsalpha} 
	\bibliography{bibliografia} 
\end{document}

%% file: acima_abaixo.pdf_tex
\begingroup%
  \makeatletter%
  \providecommand\color[2][]{%
    \errmessage{(Inkscape) Color is used for the text in Inkscape, but the package 'color.sty' is not loaded}%
    \renewcommand\color[2][]{}%
  }%
  \providecommand\transparent[1]{%
    \errmessage{(Inkscape) Transparency is used (non-zero) for the text in Inkscape, but the package 'transparent.sty' is not loaded}%
    \renewcommand\transparent[1]{}%
  }%
  \providecommand\rotatebox[2]{#2}%
  \newcommand*\fsize{\dimexpr\f@size pt\relax}%
  \newcommand*\lineheight[1]{\fontsize{\fsize}{#1\fsize}\selectfont}%
  \ifx\svgwidth\undefined%
    \setlength{\unitlength}{195.98107177bp}%
    \ifx\svgscale\undefined%
      \relax%
    \else%
      \setlength{\unitlength}{\unitlength * \real{\svgscale}}%
    \fi%
  \else%
    \setlength{\unitlength}{\svgwidth}%
  \fi%
  \global\let\svgwidth\undefined%
  \global\let\svgscale\undefined%
  \makeatother%
  \begin{picture}(1,1.00077011)%
    \lineheight{1}%
    \setlength\tabcolsep{0pt}%
    \put(0,0){\includegraphics[width=\unitlength,page=1]{acima_abaixo.pdf}}%
    \put(0.26226428,0.68767453){\color[rgb]{0,0,1}\makebox(0,0)[lb]{\smash{\footnotesize $\gamma_1$}}}%
    \put(0.2645421,0.26819126){\color[rgb]{1,0.16,0.49}\makebox(0,0)[lb]{\smash{\footnotesize $\gamma_2$}}}%
    \put(0.521,0.95){\color[rgb]{0,0,0}\makebox(0,0)[lb]{\smash{\footnotesize $\phi$}}}%
    \put(0.51393246,0.53791019){\color[rgb]{0,0,1}\makebox(0,0)[lb]{\smash{\footnotesize $\phi(t_1)=z_1$}}}%
    \put(0.51042031,0.39073869){\color[rgb]{1,0.16,0.49}\makebox(0,0)[lb]{\smash{\footnotesize $\phi(t_2)=z_2$}}}%
    \put(0.06859603,0.39034846){\color[rgb]{0,0,0}\makebox(0,0)[lb]{\smash{\footnotesize $\phi_2$}}}%
    \put(0.33303814,0.89249947){\color[rgb]{0,0,0}\makebox(0,0)[lb]{\smash{\footnotesize $\phi_1$}}}%
    \put(0.10075842,0.65786473){\color[rgb]{0,0,1}\makebox(0,0)[lb]{\smash{\footnotesize $z'_1$}}}%
    \put(0.12480678,0.26701292){\color[rgb]{1,0.16,0.49}\makebox(0,0)[lb]{\smash{\footnotesize $z'_2$}}}%
    \put(0.85,0.95){\color[rgb]{0,0,0}\makebox(0,0)[lb]{\smash{\footnotesize $\R^2$}}}%
  \end{picture}%
\endgroup%

%% file: intersecao.pdf_tex
\begingroup%
  \makeatletter%
  \providecommand\color[2][]{%
    \errmessage{(Inkscape) Color is used for the text in Inkscape, but the package 'color.sty' is not loaded}%
    \renewcommand\color[2][]{}%
  }%
  \providecommand\transparent[1]{%
    \errmessage{(Inkscape) Transparency is used (non-zero) for the text in Inkscape, but the package 'transparent.sty' is not loaded}%
    \renewcommand\transparent[1]{}%
  }%
  \providecommand\rotatebox[2]{#2}%
  \newcommand*\fsize{\dimexpr\f@size pt\relax}%
  \newcommand*\lineheight[1]{\fontsize{\fsize}{#1\fsize}\selectfont}%
  \ifx\svgwidth\undefined%
    \setlength{\unitlength}{195.9810154bp}%
    \ifx\svgscale\undefined%
      \relax%
    \else%
      \setlength{\unitlength}{\unitlength * \real{\svgscale}}%
    \fi%
  \else%
    \setlength{\unitlength}{\svgwidth}%
  \fi%
  \global\let\svgwidth\undefined%
  \global\let\svgscale\undefined%
  \makeatother%
  \begin{picture}(1,1.00036758)%
    \lineheight{1}%
    \setlength\tabcolsep{0pt}%
    \put(0,0){\includegraphics[width=\unitlength,page=1]{intersecao.pdf}}%
        \put(0.5,0.75){\color[rgb]{0,0,0}\makebox(0,0)[lb]{\smash{\footnotesize $\phi_{\tilde\gamma_1(t_1)}$}}}%
    \put(-0.23,0.589){\color[rgb]{0,0,0}\makebox(0,0)[lb]{\smash{\footnotesize $\phi_{\tilde\gamma_2(a_2)}$}}}%
    \put(-0.23,0.389){\color[rgb]{0,0,0}\makebox(0,0)[lb]{\smash{\footnotesize $\phi_{\tilde\gamma_1(a_1)}$}}}%
    \put(0.989,0.33599){\color[rgb]{0,0,0}\makebox(0,0)[lb]{\smash{\footnotesize $\phi_{\tilde\gamma_2(b_2)}$}}}%
    \put(1.01,0.585){\color[rgb]{0,0,0}\makebox(0,0)[lb]{\smash{\footnotesize $\phi_{\tilde\gamma_1(b_1)}$}}}%
    \put(0.755,0.33599){\color[rgb]{0,0,1}\makebox(0,0)[lb]{\smash{\footnotesize $\tilde\gamma_2$}}}%
    \put(0.72,0.59){\color[rgb]{1,0.16,0.49}\makebox(0,0)[lb]{\smash{\footnotesize $\tilde\gamma_1$}}}%
    \put(0.85,0.95){\color[rgb]{0,0,0}\makebox(0,0)[lb]{\smash{\footnotesize $\tildedom(\F)$}}}%
  \end{picture}%
\endgroup%

%% file: K.pdf_tex
\begingroup%
  \makeatletter%
  \providecommand\color[2][]{%
    \errmessage{(Inkscape) Color is used for the text in Inkscape, but the package 'color.sty' is not loaded}%
    \renewcommand\color[2][]{}%
  }%
  \providecommand\transparent[1]{%
    \errmessage{(Inkscape) Transparency is used (non-zero) for the text in Inkscape, but the package 'transparent.sty' is not loaded}%
    \renewcommand\transparent[1]{}%
  }%
  \providecommand\rotatebox[2]{#2}%
  \newcommand*\fsize{\dimexpr\f@size pt\relax}%
  \newcommand*\lineheight[1]{\fontsize{\fsize}{#1\fsize}\selectfont}%
  \ifx\svgwidth\undefined%
    \setlength{\unitlength}{143.05039371bp}%
    \ifx\svgscale\undefined%
      \relax%
    \else%
      \setlength{\unitlength}{\unitlength * \real{\svgscale}}%
    \fi%
  \else%
    \setlength{\unitlength}{\svgwidth}%
  \fi%
  \global\let\svgwidth\undefined%
  \global\let\svgscale\undefined%
  \makeatother%
  \begin{picture}(1,2.9879027)%
    \lineheight{1}%
    \setlength\tabcolsep{0pt}%
    \put(0,0){\includegraphics[width=\unitlength,page=1]{K.pdf}}%
    \put(-0.15,1.45){\color[rgb]{0,0,0}\makebox(0,0)[lb]{\smash{\footnotesize $\hat K$}}}%
    \put(-0.6,2.65){\color[rgb]{0,0,0}\makebox(0,0)[lb]{\smash{\footnotesize $\hat K+(0,M)$}}}%
    \put(-0.6,.25){\color[rgb]{0,0,0}\makebox(0,0)[lb]{\smash{\footnotesize $\hat K-(0,M)$}}}%
  \end{picture}%
\endgroup%